\documentclass[12pt]{amsart}
\usepackage{appendix}
\usepackage[english]{babel}
\usepackage[utf8]{inputenc}
\usepackage{amsmath, amssymb, amsthm, amsfonts, amsxtra, stmaryrd, enumerate, graphicx}
\usepackage[colorinlistoftodos]{todonotes}
\usepackage{placeins}

\newtheorem{theorem}{Theorem}[section]
\newtheorem{corollary}{Corollary}[theorem]
\newtheorem{claim}[theorem]{Claim}
\newtheorem{subclaim}[theorem]{Subclaim}
\newtheorem{lemma}[theorem]{Lemma}
\newtheorem{fact}[theorem]{Fact}
\newtheorem{proposition}[theorem]{Proposition}

\usepackage{tikz-cd}
\newtheorem{definition}{Definition}[section]

\theoremstyle{definition}

\newtheorem{example}[theorem]{Example}

\def\Ind#1#2{#1\setbox0=\hbox{$#1x$}\kern\wd0\hbox to 0pt{\hss$#1\mid$\hss}
\lower.9\ht0\hbox to 0pt{\hss$#1\smile$\hss}\kern\wd0}

\def\ind{\mathop{\mathpalette\Ind{}}}

\def\notind#1#2{#1\setbox0=\hbox{$#1x$}\kern\wd0
\hbox to 0pt{\mathchardef\nn=12854\hss$#1\nn$\kern1.4\wd0\hss}
\hbox to 0pt{\hss$#1\mid$\hss}\lower.9\ht0 \hbox to 0pt{\hss$#1\smile$\hss}\kern\wd0}

\def\nind{\mathop{\mathpalette\notind{}}}

\usepackage{etoolbox}
\patchcmd{\subsection}{-.5em}{.5em}{}{}

\title{Generic Expansions and the Group Configuration Theorem}
\author{Scott Mutchnik}

\begin{document}

\begin{abstract}
We exhibit a connection between geometric stability theory and the classification of unstable structures at the level of simplicity and the $\mathrm{NSOP}_{1}$-$\mathrm{SOP}_{3}$ gap. Particularly, we introduce generic expansions $T^{R}$ of a theory $T$ associated with a definable relation $R$ of $T$, which can consist of adding a new unary predicate or a new equivalence relation. When $T$ is weakly minimal and $R$ is a ternary fiber algebraic relation, we show that $T^{R}$ is a well-defined $\mathrm{NSOP}_{4}$ theory, and use one of the main results of geometric stability theory, the \textit{group configuration theorem} of Hrushovski, to give an exact correspondence between the geometry of $R$ and the classification-theoretic complexity of $T^{R}$. Namely, $T^{R}$ is $\mathrm{SOP}_{3}$, and $\mathrm{TP}_{2}$ exactly when $R$ is geometrically equivalent to the graph of a type-definable group operation; otherwise, $T^{R}$ is either simple (in the predicate version of $T^{R}$) or $\mathrm{NSOP}_{1}$ (in the equivalence relation version.) This gives us new examples of strictly $\mathrm{NSOP}_{1}$ theories.

\end{abstract}
\maketitle

\section{Introduction}

This paper connects two subfields of model theory: geometric stability theory and the classification theory of \textit{unstable} structures. Geometric stability theory, an excellent exposition of which is given in \cite{P96}, relates pregeometries in stable theories, such as the pregeometry defined by algebraic closure on a strongly or weakly minimal structure, to the global structure of those theories. One of the most important theorems of geometric stability theory is the \textit{group configuration theorem} of Hrushovski, which says that the incidence pattern of four lines in a projective plane, viewed entirely from within the geometric structure of the algebraic closure in a stable theory, must arise from a type-definable group:

\begin{fact}\label{gct}
    (Group Configuration Theorem, Hrushovski (\cite{Hr92})): Let $T$ be a stable theory and $a, b, c, x, y, z$ nonalgebraic tuples. Suppose, in the below Figure 1, that any three noncollinear points are independent, but any point is in the algebraic closure of any other two points on the same line. Then for some parameter set $A$ independent from $abcxyz$, there is some connected group $G$ type-definable over $A$ so that, for $a', c', x'$ independent generics of $G$ over $A$ and $b' = c'\cdot a'$, $x'=a'\cdot y'$ and $b' = z' \cdot y'$, each of $a, b, c, x, y, z$ is individually interalgebraic over $D$ with, respectively, $a', b', c', x', y', z'$. 

    \begin{figure}
        \centering

\tikzset{every picture/.style={line width=0.75pt}} 

\begin{tikzpicture}[x=0.75pt,y=0.75pt,yscale=-1,xscale=1]

\draw    (90,89) -- (39,187) ;
\draw    (90,89) -- (142,187) ;
\draw  [dash pattern={on 0.84pt off 2.51pt}]  (64.5,138) -- (142,187) ;
\draw  [dash pattern={on 0.84pt off 2.51pt}]  (116,138) -- (39,187) ;

\draw (86,69.4) node [anchor=north west][inner sep=0.75pt]    {$a$};
\draw (47,126.4) node [anchor=north west][inner sep=0.75pt]    {$b$};
\draw (121,124.4) node [anchor=north west][inner sep=0.75pt]    {$x$};
\draw (26,185.4) node [anchor=north west][inner sep=0.75pt]    {$c$};
\draw (142,184.4) node [anchor=north west][inner sep=0.75pt]    {$y$};
\draw (85,153.4) node [anchor=north west][inner sep=0.75pt]    {$z$};

\end{tikzpicture}

        \caption{The basic case of the group configuration. Diagram based on \cite{Bays15}.}
        
    \end{figure}
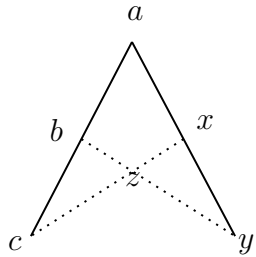
\end{fact}

This result has been generalized to some unstable contexts, such as simple theories \cite{BYTW02}, o-minimal theories \cite{Pet19}, and generically stable types \cite{Wan22}. In the following, we will show that the original group configuration theorem for \textit{stable} theories has applications to classification theory outside of the stable or even simple context.

One central question in the classification theory of unstable structures, much of which was initiated alongside the classification of stable theories by Shelah \cite{Sh90}, asks which classification-theoretic properties are equivalent and which are distinct. For example, until recently it was open whether the class $\mathrm{NSOP}_{1}$ was equal to $\mathrm{NSOP}_{2}$, and it remains open whether $\mathrm{NSOP}_{2}$ is equal to $\mathrm{NSOP}_{3}$ (\cite{DS04}); it is also open whether $\mathrm{NSOP}_{n}$ $\mathrm{NTP}_{2}$ theories are simple for $n \geq 3$ (\cite{Che14}). In the applied setting, there has also been interest in determining the classification-theoretic complexity of structures, including finding new examples of strictly $\mathrm{NSOP}_{1}$ theories. Several new examples have recently been found using generic constructions, such as algebraically closed fields of prime characteristic with a generic additive subgroup (\cite{D18}, \cite{D19}), generic incidence structures (\cite{CoK19}), generic expansions by Skolem functions (\cite{KR18}), and the $\omega$-free PAC fields (\cite{CR15}, further developed in \cite{KR17}; see also \cite{Cha99}, \cite{Cha02}), to give some examples. We will introduce the first examples, to our knowledge, where the classification-theoretic property $\mathrm{NSOP}_{1}$ results from geometric stability theory, particularly the group configuration theorem. At the boundary of $\mathrm{NSOP}_{1}$, the possible levels of classification-theoretic complexity have been characterized for theories with a suitable notion of \textit{free amalgamation}. Evans and Wong (\cite{EW09}) show that the $\omega$-categorical Hrushovski constructions introduced by Evans (\cite{Ev02}) are either simple or $\mathrm{SOP}_{3}$, and Conant (\cite{Co15}) show that modular theories satisfying some abstract free amalgmation axioms are either simple, or both $\mathrm{SOP}_{3}$ and $\mathrm{TP}_{2}$; in \cite{GFA}, the author generalizes the work of Evans and Wong, and of Conant, to potentially strictly $\mathrm{NSOP}_{1}$ theories, giving partial results on the equivalence of $\mathrm{NSOP}_{1}=\mathrm{NSOP}_{2}$ and $\mathrm{NSOP}_{3}$ covering most of the known examples of $\mathrm{NSOP}_{4}$ theories. We will introduce a family of structures defined by generic constructions, in particular the expansion of stable structures by new generic predicates or equivalence relations, whose complexity will be characterized by this $\mathrm{NSOP}_{1}-\mathrm{SOP}_{3}$ dichotomy; \textit{which} side of the dichotomy a structure in this family lies on will be characterized by the group configuration theorem.

The expansion of a theory by generic relations or function symbols was introduced by Winkler (\cite{W75}), and was studied by Chatzidakis and Pillay (\cite{CP98}) in the case of a unary predicate, which was shown to preserve simplicity. Later, Kruckman and Ramsey (\cite{KR18}) showed that expansions by generic function symbols, which covers generic equivalence relations considered as a unary function to a new sort, preserve the property $\mathrm{NSOP}_{1}$. The construction is to start with a theory $T$, add symbols to the language but no new axioms to get the theory $T_{0}$, and take the model companion $T'$ of $T_{0}$, which exists whenever $T$ eliminates quantifiers and eliminates $\exists^{\infty}$. The setting for the correspondence between groups and classification theory will be the model companion $T'$ of an expansion $T_{0}$ of a theory $T$. However, new axioms, and not just new symbols, will be added to form $T_{0}$. Allowing any axioms quickly becomes complicated, as one can encode, say, automorphisms; see \cite{BSh01}, exposited in \cite{Pil01}; \cite{KikS02}, \cite{Ki00}, for some examples of the literature on the existence of model companions of theories with automorphisms, and \cite{CH99} for a particularly interesting example.  So instead of studying all possible new axioms, we add an $n$-ary relation $R$ definable in a theory $T$ with quantifier elimination, and add universal axioms of a particular form to get a new theory $T_{0}=T_{R}$. Namely, for $P$ a 
new unary relation symbol, we add $\forall \bar{x} (\bigwedge_{1 \leq i \neq j \leq n} x_{i} \neq x_{j} \wedge \bigwedge_{i=1}^{n} P(x_{i}) \rightarrow R(\bar{x}) )$ to get $T_{R}$, or alternatively, for $E$ a new binary relation symbol, we add that $E$ is an equivalence relation and $\forall \bar{x} \bigwedge_{1 \leq i\neq j \leq n} (x_{i} \neq x_{j} \wedge E(x_{i}, x_{j})) \rightarrow R(\bar{x})$ to get $T_{R}$. When $T$ is $\mathrm{nfcp}$, $T_{R}$ in either case will then have a model companion $T'=T^{R}$. 

The main result of this paper will be on the complexity of $T^{R}$, when $T$ is weakly minimal (so $\mathrm{nfcp}$, \cite{Ga05} as observed in \cite{CL20}) and $\neg R$ is a ternary \textit{fiber algebraic relation} (\cite{CS21}, Definition 3.1). Ternary fiber algebraic relations coincide with relations of rank $\leq 2$ in the strongly minimal case, and in general they include all graphs of group operations on unary definable sets: this result says that $T^{R}$ will be classification-theoretically complicated precisely when $\neg R$ \textit{is} geometrically equivalent to the graph of a group operation:

\begin{theorem}\label{main}
Let $T$ be weakly minimal and let $R$ be a ternary relation definable in $T$. Assume $\neg R$ is fiber-algebraic. Then the equivalence relation version of $T^{R}$ is $\mathrm{NSOP}_{1}$ if and only if there is no set of parameters $A$ over which $R$ is definable, and (rank-one) group $G$ type-definable (or definable, if $T$ is strongly minimal) over $A$, so that the coordinates of a point of $\neg R$ generic (that is, of full rank) over $A$ are individually interalgebraic with the coordinates of a point of the graph $\Gamma^{G}$ of the multiplication in $G$ generic over $A$. Otherwise, $T^{R}$ is $\mathrm{TP}_{2}$ and strictly $\mathrm{NSOP}_{4}$.

For the predicate version of $T^{R}$, this is the same, but replace ``$\mathrm{NSOP}_{1}$" with ``simple."
\end{theorem}

So among ternary relations $\neg R$ that have no trivial reason \textit{not} to be the graph of a group operation, classification theory at the level of $\mathrm{NSOP}_{1}-\mathrm{SOP}_{3}$ gap measures exactly when $\neg R$ is equivalent to the graph of a (rank-one) group operation.

Geometric properties of a stable theory are known to be connected to classification-theoretic properties of its expansions. For example, \cite{Bu91} (including a result due to Hrushovski) and \cite{Va03} relate the linearity of a theory $T$ to the rank of certain generic expansions of that theory by a \textit{model} (see also \cite{Bou89} for the relationship between the \textit{dimensional order property} of a theory and the stability of its expansions by models), and \cite{BB04} shows that the pregeometry of a strongly minimal set is trivial if and only if \textit{arbitrary} expansions of that theory by a unary predicate are stable. The literature on expansions is vast--see \cite{Ba09} for an overview--and includes connections between properties of stable theories and simplicity of expansions (for example, nfcp within stable theories and simplicity in pseudo-algebraically closed expansions, \cite{Pol05}). Our result is the first that we know of to relate geometric stability theory to the classification of unstable expansions of stable theories at the level of simplicity and the $\mathrm{NSOP}_{n}$ hierarchy.

We give an outline of the paper. In section 1, we define our setting for the generic expansion $T^{R}$ of a theory $T$, associated with the definable relation $R$. The connection between the property nfcp from \cite{Sh90}, which implies stability, and axiomatizability of generic expansions, was first demonstrated by Poizat (\cite{Poi83}) in his work on \textit{belles paires} of models. This is generalized to the simple case in \cite{BYPV02} using the weaker \textit{wnfcp}, and generalized further using the nfcp in \cite{KCHY21}; see \cite{CZ01}, \cite{BB04}, \cite{La13}, \cite{BL22} for examples of connections of nfcp to more general (not necessarily generic) expansions, and \cite{Ba09} for an overview of the connections between expansions and nfcp. Generalizing the arguments of Poizat, we show that when $T$ is $\mathrm{nfcp}$, and $R$ is a definable relation, both the predicate version and the equivalence relation version of $T^{R}$ exist. We also give a converse, encoding a partial automorphism of a linear order with the construction for $T^{R}$ when $T$ is unstable and showing that the model companion cannot exist using the argument from \cite{Ki00}. This gives us a new characterization of $\mathrm{nfcp}$ of independent interest. When $\neg R$ is a fiber-algebraic ternary relation, we observe that $T^{R}$ has quantifier elimination up to finite covers.\footnote{In fact, when $\neg R$ is a fiber-algebraic ternary relation, $\mathrm{nfcp}$ is not required for $T^{R}$ to be well-defined: $T^{R}$ is well-defined even when $T$ only eliminates $\exists^{\infty}$. Because the weakly minimal case, where our main result holds, is already $\mathrm{nfcp}$ (\cite{Ga05}, \cite{Co15}), we relegate this to an appendix. We would like to acknowledge Gabriel Conant for drawing our attention to this.}

In section 2, we consider relational expansions of $\mathrm{NSOP}_{1}$ theories with quantifier elimination and general free amalgamation properties. In \cite{GFA}, the notion of \textit{Conant-independence} was introduced as the extension of the \textit{Kim-independence} from \cite{KR17} beyond the $\mathrm{NSOP}_{1}$ theories. Using results from \cite{GFA}, we characterize Conant-independence in these theories, and show that the theory is either $\mathrm{NSOP}_{1}$, or both $\mathrm{TP}_{2}$ and strictly $\mathrm{NSOP}_{4}$; the underlying arguments for the classification-theoretic results come from Conant's work in \cite{Co15}, with a new lemma of the author from \cite{NSOP2} (which can itself be proven using the proof of Proposition 3.14 of \cite{KR17}; see Footnote 1 of \cite{INDNSOP3}, and \cite{Lee22}). Meanwhile, the most general version of the result on Conant-independence will come from an improvement, very similar to \cite{KTW22}, on the ``algebraically reasonable chain condition" that was applied to $\mathrm{NSOP}_{1}$ generic expansions in \cite{KR18}. This gives us a general criterion for classifying expansions of $\mathrm{NSOP}_{1}$ theories, which will be applied to the particular case of $T^{R}$ where $T$ is weakly minimal and $\neg R$ is a fiber-algebraic ternary relation.

Finally, in section $3$, we prove our main result, Theorem \ref{main}.

Notations are standard. We use $\bar{x}, \bar{y}, \bar{z}$, etc. and $\bar{a}, \bar{b}, \bar{c}$ to denote tuples of variables or constants, and $x, y, z$, $a, b, c$ to denote tuples or singletons depending on context.

\section{The model companion}

We define the general setting for this section and section 4. Throughout this paper, the theory $T$ will always have quantifier elimination in the language $\mathcal{L}$.

We start with the predicate version of this setting. Let $\delta(\bar{x})$ denote that the coordinates of $\bar{x}$ are distinct. Let $P$ be an additional unary predicate symbol and $\mathcal{L}^{P}=\mathcal{L} \cup \{P\}$. Let $R$ be a definable $n$-ary relation in $\mathcal{L}$. Define $T_{R}$ to be the $\mathcal{L}^{P}$-theory consisting of $T$ together with the axiom $\forall \bar{x} (\delta(\bar{x}) \wedge \bigwedge_{i=1}^{n} P(x_{i}) \rightarrow R(\bar{x}) )$. In words, $T_{R}$ will be the theory of models of $T$ together with a unary predicate $P$ so that any $n$-tuple with distinct coordnates in $P$ will belong to $R$. If $T_{R}$ has a model companion, we denote it $T^{R}$; in analogy with \cite{CP98}, $T^{R}$ will be the generic expansion of $T$ by a unary predicate, subject to a universal constraint.

Now we define the equivalence relation version of the setting. Let $E$ be an additional binary relation symbol and $\mathcal{L}^{E} = \mathcal{L} \cup \{E\}$. Let $R$ be a definable $n$-ary relation in $\mathcal{L}$. Define $T_{R}$ to be the $\mathcal{L}^{E}$-theory consisting of the axioms for $T$, the requirement that $E$ be an equivalence relation, and the axiom $\forall \bar{x} \bigwedge_{1 \leq i\neq j \leq n} (x_{i} \neq x_{j} \wedge E(x_{i}, x_{j})) \rightarrow R(\bar{x})$. In words, $T_{R}$ will be the theory of models of $T$ together with an equivalence relations $E$ so that any $n$-tuple of distinct elements of the same equivalence class will belong to $R$. If $T_{R}$ has a model companion, we denote it $T^{R}$.

The predicate and equivalence relation version of $T^{R}$ will have the same properties, except that the equivalence relation version can be strictly $\mathrm{NSOP}_{1}$, and the proofs for each version will be similar. When it is not clear from context, we will use $T_{R, P}$, $T^{R}_{P}$ to denote the predicate version and $T_{R, E}$, $T^{R}_{E}$ to denote the equivalence relation version.

We would like to know when $T^{R}$ exists. In fact, we characterize the theories that can only interpret theories $T$, so that $T^{R}$ always exists for any $R$ definable in $T$. We need the following classification theoretic property, from \cite{Sh90}:

\begin{definition}
    A formula $\varphi(x, y)$ has \emph{nfcp}, or the \emph{non-finite cover property}, if there is some $n$ so that any set $\{\varphi(x, b_{i})\}_{i \in I}$ is consistent if and only if it is $n$-consistent. A theory is \emph{nfcp} if every formula is nfcp.
\end{definition}

The following generalizes the direction (i) $\rightarrow$ (ii) of Theorem 6 of \cite{Poi01}; as expected, it uses the fact that consistency of a $\varphi$-type is definable in an nfcp theory.

\begin{proposition} \label{mc}
Let $T$ be an nfcp theory, and $R$ an $n$-ary relation definable in $T$. Let $\mathcal{L}$ be the language $\mathcal{L}_{0}$ of $T$ together with an additional symbol $P$ for a unary relation. Let $T_{R}$ be the theory in $\mathcal{L}$ of models $M \models T$ such that, for any $n$-tuple $\bar{a} \in P(M)$, $M \models R(\bar{a})$. Then the model companion $T^{R}$ of $T_{R}$ exists.
\end{proposition}

\begin{proof}
(\textit{Equivalence relation version.})

The theory $T_{R}$ is a consistent theory: isolate each element of a model in its own $E$-equivalence class. (This was the purpose of requiring that the $x_{i}$ be distinct.) Since $T_{R}$ is formed from the model-complete theory $T$ by adding universal axioms, ascending chains of models of $T$ are again models of $T$. It follows that existentially closed models  of $T_{R}$ exist, and it remains to show that the class of existentially closed models is axiomatizable. Let $\varphi(\bar{y}, \bar{x})$ be a quantifier-free $\mathcal{L}^{E}$-formula. We will show that there is an $\mathcal{L}^{E}$-formula $\rho(y)$ such that, for $M \models T_{R}$, $M \models \rho(\bar{m})$ if and only if there is some extension $N \supseteq M$ with $N \models T_{R}$ so that $N \models \varphi(\bar{m}, \bar{n})$ for some $\bar{n} \in N$ with $\bar{n} \cap M = \emptyset$. This will be enough for us, by the following claim:

\begin{claim} 
Suppose that for each quantifier-free $\mathcal{L}^{E}$-formula $\varphi(\bar{x}, \bar{y})$, there exists a $\mathcal{L}^{E}$-formula $\rho_{\varphi}(\bar{y})$ as above. Then the sentences $\forall y (\rho_{\varphi}(\bar{y}) \rightarrow \exists \bar{x} \varphi(\bar{y}, \bar{x}))$ where $\varphi(\bar{y}, \bar{x})$ ranges over the quantifier-free $\mathcal{L}^{P}$-formulas, will axiomatize when a model $M \models T_{R}$ is existentially closed.
\end{claim}

\begin{proof}
    (Implicit in the proof of Theorem 2.4 of \cite{CP98}.) Clearly, an existentially closed model of $T_{R}$ satisfies these sentences. Conversely, let $M \models T_{R}$ satisfy these sentences. We show that $M$ is existentially closed. It suffices to show that for $\psi(\bar{y}; \bar{z} \bar{x})$ a quantifier-free $\mathcal{L}^{P}$-formula, $ M \subseteq N \models T_{R}$, $\bar{m}, \bar{a} \in M$, $\bar{b} \in N$ with $\bar{b} \cap N = \emptyset$, and $N\models \psi(\bar{m}, \bar{a}, \bar{b})$, there is some $\bar{b'} \in M$ so that $M\models \psi(\bar{m}, \bar{a}, \bar{b'})$. We just apply the hypothesis to $\varphi(\bar{y}\bar{z}; \bar{x}) =: \psi(y, zx)$, noting $M\models \rho_{\varphi}(\bar{m}\bar{a})$ in this case.
\end{proof}

Our strategy will be as follows. The existence of $ N \supseteq M$, with $\bar{n} \in N \models T_{R}$ and $\bar{n} \cap M = \emptyset$, so that $N \models \varphi(\bar{m}, \bar{n})$ will be equivalent to the consistency of a partial $\mathcal{L}$-type, consisting of instances of finitely many $\mathcal{L}$-formulas, where the parameters for those instances can be described in the language $\mathcal{L}^{P}$ in a way that is uniform in $\bar{m}$ and $M$. As the $\mathcal{L}$-formulas are $\mathrm{nfcp}$, the consistency of this type is equivalent to $n$-consistency of the type, which can be expressed by an $\mathcal{L}^{P}$-formula.

We may assume that $\varphi(\bar{y}, \bar{x})$ is a formula of the form $\psi(\bar{y}, \bar{x}) \wedge \eta(\bar{y}, \bar{x}) $ where 

\begin{itemize}
\item$\psi(\bar{y}, \bar{x})$ is a $\mathcal{L}$-formula that implies $\bar{y}, \bar{x}$ is a tuple of distinct elements, and

\item $\eta(\bar{y}, \bar{x})$ is a consistent boolean combination of instances of $E(x_{i}, y_{j})$ that completely describes the restriction of the equivalence relation $E$ to the variables $\bar{y}, \bar{x}$.
\end{itemize}

We will define a $\mathcal{L}$-formula $\tau(\bar{x}, \bar{y}, \bar{z})$. Here $\bar{z} = (z_{1}, \ldots, z_{n}, \bar{w}_{1}, \ldots, \bar{w}_{n})$, and $\bar{w}_{i}$ is an $N+1$-tuple of variables, where $C_{1}, \ldots C_{N}$ is an enumeration of the equivalence classes on the variables $\bar{y}, \bar{x}$ described by the formula $\eta(\bar{y}, \bar{x})$ and containing variables from $\bar{y}$. Let $\tau(\bar{x}, \bar{y}, \bar{z})$ express that

(a) $\models \psi(\bar{y}, \bar{x})$

(b) Let $\bar{a}$ be an $n$-tuple with distinct coordinates drawn from $\bar{x}$, whose coordinates are required by $\eta(\bar{y}, \bar{x})$ to belong to the same equivalence class, not containing any of the $\bar{y}$. Then $\models R(\bar{a})$

(c) $\bigwedge_{i\leq |\bar{x}|, j\leq n} x_{i} \neq z_{j}$

(d) Let $\bar{a}$ be an $n$-tuple, consisting of $\bar{x}$-coordinates all required by $\eta(\bar{y}, \bar{x})$ to belong to the equivalence class represented by $C_{i}$ for some fixed $i$, and $\bar{z}$ coordinates of the form $\bar{z}_{j}$ such that $\bar{w}_{j}$ consists of exactly $i$ many distinct elements. (So the $\bar{w}_{j}$ encode the indices of the $C_{1}, \ldots C_{N}$). Then $\models R(\bar{a})$.

Note that $\tau(\bar{x}, \bar{y}, \bar{z})$ can indeed be chosen to be a $\mathcal{L}$-formula, and not a $\mathcal{L}^{P}$-formula, because the equivalence relation $E$ itself is not referred to, only the requirements imposed by $\eta(\bar{y}, \bar{x})$. Note also that it can be chosen uniformly in $M$.

For $\bar{e} \in M$ with $|\bar{e}|=|\bar{y}|$, define a partial $\tau(\bar{x}, \bar{y} \bar{z})$-type $p(x, \bar{e})$ in the variables $\bar{x}$, with parameters in $M$, as follows:

Let $p(x, \bar{e})$ be the set of $\tau(\bar{x}, \bar{e}, \bar{b})$, where $(c_{1}, \ldots, c_{n}, \bar{d}_{1}, \ldots, \bar{d}_{n})=\bar{b} \in M$, and for $ j \leq n$, $\models E(c_{j}, e)$, where $e$ is any element of $\bar{e}$ required by $\eta(\bar{y}, \bar{x})$ to belong to the equivalence class $C_{i}$ on $\bar{y}, \bar{x}$, if and only if $\bar{d}_{j}$ consists of $i$ many distinct elements, and $\models  \wedge_{e \in \bar{e}}\neg E(c_{j}, e) $ if and only if $\bar{d}_{j}$ consists of $N+1$ many distinct elements.

In this form, it can be seen that for any $k$, there is an $\mathcal{L}^{P}$-formula $\rho_{k}(\bar{y})$ so that $\models \rho_{k}(\bar{e})$ if and only if $p(x, \bar{e})$ is $k$-consistent. But since $\tau(\bar{x}, \bar{y}\bar{z})$ is nfcp, there is some $k$ so that a $\tau$-type is consistent if and only if it is $k$-consistent. So there is $\rho(\bar{y})$ so that $\models \rho_{k}(\bar{e})$ if and only if $p(x, \bar{e})$ is consistent. Note that $\rho(\bar{y})$ can be chosen uniformly in $M$. We show $\rho(\bar{y})$ is as desired.

Note that $p(x, \bar{e})$ expressed the following conditions

(a$'$) $\models \psi(\bar{e}, \bar{x})$

(b$'$) Let $\bar{a}$ be an $n$-tuple with distinct coordinates drawn from $\bar{x}$, whose coordinates are required by $\eta(\bar{y}, \bar{x})$ to belong to the same equivalence class, not containing any of the $\bar{y}$. Then $\models R(\bar{a})$

(c$'$) $\bar{x} \cap M =\emptyset$ 

(d$'$) Let $\bar{a}$ be an $n$-tuple with distinct coordinates drawn from $M \cup \bar{x}$, consisting of $\bar{x}$-coordinates all required by $\eta(\bar{e}, \bar{x})$ to belong to the equivalence class of $e \in \bar{e}$, and elements of $M$ belonging to the equivalence class of $e$. Then $\models R(\bar{a})$.

It remains to show that the following are equivalent:

(i) there exists $ N \supseteq M$, with $\bar{n} \in N \models T_{R}$ and $\bar{n} \cap M = \emptyset$, so that $N \models \varphi(\bar{e}, \bar{n})$

(ii) $p(\bar{x}, \bar{e})$ is consistent

For (i $\Rightarrow$ ii), clearly $\bar{n}\models p(x, \bar{e})$. Conversely, let $M \prec N \models T$ be an $\mathcal{L}$-elementary of $N$, and $\bar{n}\models p(x, \bar{e})$. By (c'), $\bar{n} \cap M = 0$. Therefore, we can expand $M$ to a $\mathcal{L}^{E}$-extension of $M$ as follows: choose the finest equivalence relation on $N$ extending that on $M$, so that $N \models \eta(\bar{e}, \bar{n})$. By (a') and the fact that $\varphi(\bar{y}, \bar{e}) = \psi(\bar{y}, \bar{x}) \wedge \eta(\bar{y}, \bar{x})$,  $N \models \varphi(\bar{e}, \bar{n})$. To complete the proof of (ii) $\Rightarrow$ (i) it remains to show $N \models T^{R}$. Any elements of $\bar{n}$ equivalent to elements of $M$ are equivalent to some $e \in \bar{e}$, while any elements of $N \backslash (M \cup \bar{n})$ are isolated. So to show show that $N \models T^{R}$, we just need to show $N \models R(\bar{a})$ for $\bar{a}$ a tuple of distinct elements taken from an equivalence class of some $a \in \bar{a}$ not equivalent to any element of $M$, or from an equivalence class of $e \in \bar{e}$. The first case follows from (b'), and the second from (d').

(\textit{Predicate version.}) This is similar to the equivalence relation version, but less complicated, so we only give a sketch.

As in the equivalence relation version, it suffices to find $\rho(\bar{y})$. By the claim, it suffices to define those $\bar{m}$ so that $\varphi(\bar{m}, \bar{x})$ is realized in some model of $T_{R}$ extending $M$, by a tuple with no coordinates in $M$. We may assume that $\varphi(\bar{y}, \bar{x})$ is of the form $\psi (\bar{y}, x_{1}, \ldots x_{m}) \wedge \bigwedge_{1}^{k}P(x_{i}) \wedge \bigwedge_{k+1}^{m} \neg P(x_{i}) $ for $\psi (\bar{y}, x_{1}, \ldots x_{m})$ an $\mathcal{L}$-formula. Let $\tau(\bar{x}, \bar{y}, \bar{z})$ where $\bar{z} =  (z_{0}, z_{1}, \ldots, z_{n}, \bar{w}_{0}, \bar{w}_{1}, \ldots, \bar{w}_{n})$ and  $\bar{w}_{i}=(u_{i}, v_{i})$ be the $\mathcal{L}_{0}$-formula expressing the following: $\psi (\bar{y}, x_{1}, \ldots x_{m})$ is true, $\bigwedge_{i\leq |\bar{x}|, j\leq n} x_{i} \neq z_{j}$, and for each of the $n$-tuples $\bar{a}$ of distinct elements whose coordinates are among the $x_{1}, \ldots x_{k}$ together with those $z_{i}$, $1 \leq i \leq n$ with $u_{i} = v_{i}$, $M \models R(\bar{a})$. For $\bar{e} \in M$, let the set $p(\bar{x}, \bar{y})$ consist of those $\bar{z}$-instances of $\tau(\bar{y}, \bar{x}, \bar{z})$ over $M$ for which $z_{i} \in P(M)$ if and only if $u_{i} = v_{i}$. Then (i) $\varphi(\bar{e}, \bar{x})$ is realized in an extension of $M$ by a tuple with no coordinates in $M$ if and only if (ii) $p(\bar{x}, \bar{e})$ is consistent. For the direction (ii) $\Rightarrow$ (i), proceed as in the equivalence relation version, but choose $P(N) = P(M) \cup \bar{\bar{n}}$.  Because $T$ is nfcp, there is some $k$ not depending on $M$ so that, for $m \in M$, any family of $\bar{z}$-instances of $\tau(\bar{m}, \bar{y}, \bar{z})$ is consistent if it is $k$-consistent. But $k$-consistency of $p(\bar{m}, \bar{x})$ is expressible in $\mathcal{L}^{P}$.

\end{proof}

Note that nfcp is preserved under interpretation. Therefore, if $T$ is $\mathrm{nfcp}$, any theory that is interpretable in $T$ satisfies the conclusion of the previous proposition. We show a converse, which may be of independent interest. That is, if $T$ is not nfcp, then it defines a theory $T'$ such that, for some $T'$-definable relation $R$, $T'^{R}$ is not well-defined. Since this converse is not necessary for our main results, we will focus on proving the predicate version. The following is Theorems II.4.2 and II.4.4 of \cite{Sh90}:

\begin{fact}
    A theory $T$ is nfcp if and only if $T$ is stable and $T^{\mathrm{eq}}$ eliminates $\exists^{\infty}$. 
\end{fact}

\begin{lemma}
A theory is nfcp if and only if it is stable and, for any formula $\varphi(x, y, z)$ (where $x, y, z$ are tuples of variables) and $M \models T$, the set of $a \in M$ so that $\{\varphi(x, m, a): m \in \mathbb{M} \}$ is consistent, is definable.

\end{lemma}

\begin{proof}
    ($\Rightarrow$) The property $\mathrm{nfcp}$ implies stability by the previous fact. It is immediate from $\mathrm{nfcp}$ that there is some $k$ so that for $a \in \mathbb{M}$, $\{\varphi(x, m, a): m \in \mathbb{M} \}$ is consistent if and only if it is is $k$-consistent, so consistency of this set is in fact definable.

    ($\Leftarrow$) By the previous fact, it suffices to show that $T^{eq}$ eliminates $\exists^{\infty}$.  Let $E$ be a definable equivalence relation on tuples $x$ and $\psi(x, a)$ a formula in $T$; it suffices to show that having infinitely many $E$-inequivalent realizations $x$ of $\psi(x, b)$ is a definable property of $b$. For $\varphi(x, y, z) =: \psi(x, z) \wedge (\psi(y, z) \rightarrow \neg(x E y))$, and $b \in \mathbb{M}$, consistency of $\{\varphi(x, m, b): m \in \mathbb{M}\}$ is the same thing as saying that, for any finite collection of realizations of $\psi(x, b)$, there is some realization that is $E$-inequivalent to any realization in this collection. This is of course the same thing as $\psi(x, b)$ having infinitely many $E$-inequivalent realizations. But by the assumption, consistency of $\{\varphi(x, m, b): m \in \mathbb{M}\}$ is a definable property of $b$. 
\end{proof}

The following generalizes the arguments from (iv) $\rightarrow$ (i) of Theorem 6 of \cite{Poi01}, as well as Proposition 2.11 of \cite{CP98}. 

\begin{lemma}
Suppose that in $T$ it fails that for any formula $\varphi(x, y, z)$, the set of $a$ so that $\{\varphi(x, m, a): m \in \mathbb{M} \}$ is consistent, is definable. Then there is a theory $T'$ definable in $T$, and $R$ definable in $T'$, so that $T'^{R}$ does not exist.
\end{lemma}

\begin{proof}
    (Predicate version) Our strategy will be to encode a definable family of families of sets, containing families that are $n$-consistent but not $n+1$ consistent for each $n$, as a unary predicate. Let $T'$ be the theory of models $M \models T$ together with an additional sort $S_{1}$ for pairs $(m_{1}, m_{3}) \in M$ and another sort $S_{2}$ for singletons $m_{2} \in M$, together with the same definable relations as in $T^{\mathrm{eq}}$. (So $S_{1}$ can be identified with the set of pairs of elements of the home sort.) Let $\varphi(x, y, z)$ witness the failure of the above property, and let $R(x, y)$ express that if $x = (m_{1}, m_{3}) \in S_{1}$ and $y = m_{2} \in S_{2}$ that $M \models \varphi(m_{1}, m_{2}, m_{3}) $. We show that $T'^{R}$ is not well-defined.

    Suppose it is well-defined. Let $M'_{n} \models T'$ and $b_{n} \in M'_{n}$ (the home sort) so that $\{\varphi(x, m, b_{n}): m \in M'_{n} \}$ is $n$-consistent but not $n+1$-consistent. Expand $M'_{n}$ to a model of $T_{R}$ so that $\{\varphi(x, m, b_{n}): m \in P(M'_{n}) \}$ is $n+1$-inconsistent. Find an extension $M_{n} \supseteq M'_{n}$ with $M_{n} \models T^{R}$; then for $P_{n}$ the $S_{2}$-points of $P(M_{n})$, $\{\varphi(x, m, b_{n}): m \in P_{n} \}$ is $n$-consistent but not $n+1$-consistent. In fact, any $\lfloor \frac{n}{2} \rfloor$ formulas of $\{\varphi(x, m, b_{n}): m \in P_{n} \}$ must have at least $\frac{n}{2}$ realizations, as the other formulas can be used to distinguish the realizations\footnote{A similar observation was made in the proof of Theorem 7.3 of \cite{Che14}, also as a strategy of getting infinitely many realizations for a set of formulas.}. By the $n+1$-inconsistency, there can be no point of the form $(x, b_{n}) \in S_{1}$ in $P(M_{n})$. So by compactness, we can find a model $M \models T^{R}$ and $b \in M$ so that $\{\varphi(x, m, b): m \in M \}$ has infinitely many realizations, but there is no point of the form $(x, b) \in S_{1}$ in $P(M)$. This is a contradiction, since $M$ is existentially closed; we can in fact find a point $a \notin M$ realizing $\{\varphi(x, m, b): m \in M \}$ in an $\mathcal{L}$-elementary extension and label $(a, b)$ with $P$ to get a model of $T_{R}$.

    (Equivalence relation version) This is essentially the same proof; just instead of considering the domain of $P$, get $b_{n}, M_{n}$ and a particular equivalence class in place of $P(M_{n})$ with the desired properties, then apply compactness so that there is $b, M$ and a new equivalence class in place of $P(M)$ that gives us a contradiction.

\end{proof}

 To characterize nfcp \textit{within general theories} in terms of generic structures, we show the following:

\begin{lemma}
Suppose that $T$ is unstable.
Then there is a theory $T'$ definable in $T$, and $R$ definable in $T'$, so that $T'^{R}$ does not exist.
\end{lemma}

One obstruction to the existence of a model companion is automorphisms of an ordered set. This strategy was used in \cite{Ki00} to show that the theory of a (necessarily $\mathrm{SOP}$) unstable $\mathrm{NIP}$ structure with an automorphism did not have a model companion; then the result was improved to the general $\mathrm{SOP}$ case in \cite{KikS02}. We follow the arguments from these papers.

\begin{proof}

(Predicate version) Let $\varphi(x, y)$ be unstable, and suppose without loss of generality that there is an indiscernible sequence of singletons $\{a_{i}\}_{i \in \omega}$ so that $\models \varphi(a_{i}, a_{j})$ if and only if $i \leq j$. Consider the theory $T_{*}$ of models $M \models T$ expanded by a unary predicate $P$ whose domain is linearly ordered by $\varphi(x, y)$ together with a binary relation $R(M) \subseteq M^{2}$ so that the restriction of $R$ to $P(M) \times P(M)$ is the graph of a partial order-automrphism of $P(M)$ with respect to $\varphi(x, y)$. In a suitable power $T'$ of $T$, the requirements on this structure can all be encoded by a universal axiom of the form  $\forall \bar{x} (\delta(\bar{x}) \wedge \bigwedge_{i=1}^{n} P'(x_{i})) \rightarrow R'(\bar{x}) $ for $R'$ some $\mathcal{L}$-definable relation and $P'$ a unary predicate (in the power) representing the additional structure. Therefore, $T'_{R}$ will be interdefinable with $T_{*}$ in a natural way, so it suffices to show that $T^{*}$ does not have a model companion. 

Suppose this model companion $T^{*}$ exists. We argue as in the proof of Theorem 3.1 of \cite{Ki00}. Let $\sigma$ be the partial order-automorphism defined on $P^{2}$ by $R$. Let $\mathcal{L}^{*}_{c}$ be the language of $T^{*}$ together with additional constant symbols $\{c_{i}\}_{i \in \omega+1}$. Let $T^{*}_{c}$ be the $\mathcal{L}^{*}_{c}$-theory formed from $T^{*}$ by additionally requiring that

(a) $\{c_{i}\}_{i \in \omega+1}$ be an $\mathcal{L}$-indiscernible $\varphi(x, y)$-increasing sequence within $P$, and

(b) $\sigma(c_{i})$ is (defined and) equal to $c_{i+1}$ for $i \in \omega$.

First, we see that $T^{*}_{c} \vdash \exists x (P(x) \wedge \varphi(c_{0}, x ) \wedge \varphi(x, c_{\omega}) \wedge \sigma(x) = x) $.

To see this, let $M$ be a model of $T^{*}_{c}$. In an $\mathcal{L}$-elementary extension $M'$, find by indiscernibility some $c$ greater, in the sense of $\varphi(x, y)$, than all of the $c_{i}$ for $i \in \omega$, but less than any element of $P(M)$ greater than all of the $c_{i}$ for $i \in \omega$. In other words, fill the right cut of $P(M)$ determined by the increasing sequence $\{c_{i}\}_{i \in \omega}$. Extend the additional structure only to declare that $c \in P(M')$ and $\sigma(c)$ is defined and equal to $c$; then since the cut determined by the increasing sequence $\{c_{i}\}_{i \in \omega}$ must be closed under $\sigma$ where it is defined on $P(M)$, $\sigma$ remains a partial order-automorphism on $P(M')$. So $M \models \exists x (P(x) \wedge \varphi(c_{0}, x ) \wedge \varphi(x, c_{\omega}) \wedge \sigma(x) = x)$ by existential closedness.

Let $(T^{*}_{c})_{n}$ be the theory $T^{*}$ together with

(a)$_n$ $\{c_{i}\}_{i \in [n] \cup \{\omega\}}$ is an $\mathcal{L}$-indiscernible $\varphi(x, y)$-increasing sequence within $P$, and 

(b)$_n$ $\sigma(c_{i})$ is defined and equal to $c_{i+1}$ for $i < n$.

By compactness, $(T^{*}_{c})_{n} \vdash \exists x (P(x) \wedge \varphi(c_{0}, x ) \wedge \varphi(x, c_{\omega}) \wedge \sigma(x) = x)$ for some $n$. For a contradiction, it remains to construct a model of $(T^{*}_{c})_{n}$ where $\exists x (P(x) \wedge \varphi(c_{0}, x ) \wedge \varphi(x, c_{\omega}) \wedge \sigma(x) = x)$ is not satisfied. But it is easy to construct a model of $(T^{*}_{c})_{n}$ where, in addition to these requirements, $\sigma(c_{n})$ is defined and equal to $c_{\omega}$. And this cannot satisfy  $\exists x (P(x) \wedge \varphi(c_{0}, x ) \wedge \varphi(x, c_{\omega}) \wedge \sigma(x) = x)$.

(Equivalence relation version) Similar to the predicate version; in the above $T'_{R}$, the information of $R, P$ is now encoded as a particular $E$-equivalence class, which we can distinguish by selecting a representative.

\end{proof}

Combining Proposition 2.1 with lemmas 2.3 through 2.5, we characterize existence of these model companions as a classification-theoretic dividing line:

\begin{theorem}
A theory $T$ is $\mathrm{nfcp}$ if and only if for every theory $T'$ definable (interpretable) in $T$ and relation $R$ definable in $T'$, $(T')^{R}$ exists.  Otherwise it is $\mathrm{fcp}$.
\end{theorem}

 Poizat, in his analysis of belles paires in \cite{Poi01}, treats only the stable case. Not only does our result generalize those of Poizat to $T^{R}$; it gives a full characterization of $\mathrm{nfcp}$ in terms of model companions, ruling out the unstable case.

We now return to the set-up for our main result, \ref{main}. We are interested in using the complexity of $T^{R}$ to classify when $\neg R$ is geometrically equivalent to the graph of a group operation in the sense described in the statement of \ref{main}. (The negation is required for this construction to be nontrivial.) The following class of ternary relations, first defined in \cite{CS21} for general $n$-ary relations, includes all relations  without no trivial reason \textit{not} to arise from a group in this sense.

\begin{definition}(\cite{CS21})
A definable ternary relation  $R$ is \emph{fiber-algebraic} if whenever $\models R(\bar{a})$, any coordinate of $\bar{a}$ is algebraic over the other two. 
\end{definition}

When $\neg R$ is a fiber-algebraic definable ternary relation, $T^{R}$ is well-defined even when $T$ eliminates $\exists^{\infty}$, and is not necessarily $\mathrm{nfcp}$. Since the weakly minimal theories considered in our main result are $\mathrm{nfcp}$ (\cite{Ga05}, as observed by \cite{CL20}), we will show this in the appendix.

It is essential to our main result that $T^{R}$ admit quantifier elimimination up to finite covers. When $\neg R$ is a fiber-algebraic ternary relation, this is easy.

\begin{lemma}\label{qe1}
Let the $\mathcal{L}$-definable ternary relation $R$ be such that $\neg R$ fiber-algebraic, and let $C \subseteq A, B$ be substructures of models of $T_{R}$, algebraically closed in the sense of $\mathcal{L}$. Then there is a model $D$ of $T_{R}$ containing $A$ and $B$ as substructures, with $A \cap B = C$.
\end{lemma}

\begin{proof}
We may assume that (the reducts to $\mathcal{L}$ of) $A$, $B$ and $C$ are substructures of some model $M \models T$, with $A \cap B = C$. In the predicate version, we expand $M$ to a $\mathcal{L}$-structure extending $A$ and $B$ as follows: $P(M)=P(A) \cup P(B)$. In the equivalence relation version, we let $E(M)$ be the finest equivalence relation extending $E(A)$ and $E(B)$, so $A/E \cap B/E = C/E$ and each element of $M \backslash (A \cup B)$ is isolated. It remains to show $M \models T_{R}$.

We can assume without loss of generality that $\bar{m}$ is a triple with one coordinate in $B \backslash C$ and the other two in $A$ and must show that $M \models R(\bar{m})$ if the coordinates of $\bar{m}$ belong to $P(M)$, or to a single $E(M)$-equivalence class. But this is clear, as the one coordinate cannot be algebraic over the other two.

(Note that when an element is isolated in its own $E$-class as an element of a set, when that set is identified with a subset of a model of $T^{R}$, it is not isolated within that model.)
\end{proof}

\begin{proposition}\label{qe2}
Let the $\mathcal{L}$-definable ternary relation $R$ be such that $\neg R$ fiber-algebraic, and assume the equivalence relation version of $T^{R}$ is well-defined. Let $A$ and $B$ be substructures of models of $T^{R}$, algebraically closed in the sense of $\mathcal{L}$. Then if $\mathrm{qftp}_{\mathcal{L^{E}}}(A) = \mathrm{qftp}_{\mathcal{L^{E}}}(B)$, $\mathrm{tp}_{\mathcal{L^{E}}}(A) = \mathrm{tp}_{\mathcal{L^{E}}}(B)$. The algebraic closure in the sense of $\mathcal{L}$ and $\mathcal{L}^{E}$ coincide, so $T^{R}$ has quantifier elimination up to finite covers and the completions of $T^{R}$ are determined by $E(\mathrm{acl}(\emptyset))$. 

The same holds for the predicate version, replacing $\mathcal{L}^{E}$ with $\mathcal{L}^{P}$
\end{proposition}

\begin{proof}
Follows from the previous lemma by the standard arguments.
\end{proof}

\section{Conant-independence}

We first give an overview of classification theory beyond the simple theories; some of this discussion will be for motivation. We will consider relational expansions of the \textit{$\mathrm{NSOP}_{1}$ theories} first formally introduced in \cite{DS04}, a class which contains all simple theories.

\begin{definition}
A theory $T$ is $\mathrm{NSOP}_{1}$ if there does not exist a formula $\varphi(x, y)$ and tuples $\{b_{\eta}\}_{\eta \in 2^{<\omega}}$ so that $\{\varphi(x, b_{\sigma \upharpoonleft n})\}_{n \in \omega}$ is consistent for any $\sigma \in 2^{\omega}$, but for any $\eta_{2} \unrhd \eta_{1} \smallfrown \langle 0\rangle$, $\{\varphi(x, b_{\eta_{2}}), \varphi(x, b_{\eta_{1} \smallfrown \langle 1\rangle})\}$ is inconsistent. Otherwise it is $\mathrm{SOP}_{1}$.
\end{definition}

The main stability-theoretic tool for studying $\mathrm{NSOP}_{1}$ theories is \textit{Kim-independence}, a notion introduced by Kaplan and Ramsey (\cite{KR17}) that coincides with forking-independence in simple theories. Recall that a global type $p(x)$ is \textit{invariant} over a model $M$ if whether $\varphi(x, b)$ belongs to $p$ for $\varphi(x, y)$ a fixed formula without parameters depends only on the type of the parameter $b$ over $M$ and not on the specific realization of that type, and that an infinite sequence $\{b_{i}\}_{i \in I}$, is an \textit{invariant Morley sequence} over $M$ if there is a fixed global type $p(x)$ invariant over $M$ so that $b_{i} \models p(x)|_{M\{b_{j}\}_{j < i}}$ for $i \in I$.

\begin{definition}
A formula $\varphi(x, b)$ \emph{Kim-divides} over $M$ if there is an invariant Morley sequence $\{b_{i}\}_{i \in \omega}$ starting with $b$ (said to \emph{witness} the Kim-dividing) so that $\{\varphi(x, b_{i})\}_{i \in \omega}$ is inconsistent. A formula  $\varphi(x, b)$ \emph{Kim-forks} over $M$ if it implies a (finite) disjunction of formulas Kim-dividing over $M$. We write $a \ind^{K}_{M} b$, and say that $a$ is \emph{Kim-independent} from $b$ over $M$ if $\mathrm{tp}(a/Mb)$ does not include any formulas Kim-forking over $M$.
\end{definition}

For example, in the algebraically closed fields with a generic additive subgroup $G$ from \cite{D18}, $A\ind^{K}_{M}B$ is given by the ``weak independence" $A \ind^{\mathrm{ACF}}_{M} B$ and $G(\mathrm{acl}(MA)+\mathrm{acl}(MB))=G(\mathrm{acl}(MA))+G(\mathrm{acl}(MB))$. Analogously to simplicity, there is the following characterization of $\mathrm{NSOP}_{1}$ theories:

\begin{fact}\label{kimlemma}
(\cite{KR17}) Let $T$ be $\mathrm{NSOP}_{1}$. Then for any formula $\varphi(x,b)$ Kim-dividing over $M$, any invariant Morley sequence over $M$ starting with $b$ witnesses Kim-dividing of $\varphi(x, b)$ over $M$. Conversely, suppose that for any formula $\varphi(x,b)$ Kim-dividing over $M$, any invariant Morley sequence over $M$ starting with $b$ witnesses Kim-dividing of $b$ over $M$. Then $T$ is $\mathrm{NSOP}_{1}$.

It follows that Kim-forking coincides with Kim-dividing in any $\mathrm{NSOP}_{1}$ theory.

\end{fact}

It is standard (Proposition 3.20 of \cite{KR17}) that $\ind^{K}$ satisfies extension: if $M \subseteq B \subseteq C$ and $a \ind^{K}_{M} B$ then there is $a' \equiv_{B} a$ with $a' \ind^{K}_{M} C$.

\begin{fact}\label{symm}
(\cite{CR15}, \cite{KR17}) The theory $T$ is $\mathrm{NSOP}_{1}$ if and only if $\ind^{K}$ is a symmetric relation over models.
\end{fact}

\begin{definition}
Let $n \geq 3$. A theory $T$ is $\mathrm{NSOP}_{n}$ (that is, does not have the \emph{n-strong order property}) if there is no definable relation $R(x_{1}, x_{2})$ with no $n$-cycles, but with tuples $\{a_{i}\}_{i \in \omega}$ with $\models R(a_{i}, a_{j})$ for $i <j$. Otherwise it is $\mathrm{SOP}_{n}$.
\end{definition}

Note that $\mathrm{NSOP}_{1} \subseteq \mathrm{NSOP}_{3}$ and for $3 \leq n < m$, $\mathrm{NSOP}_{n} \subsetneq \mathrm{NSOP}_{n}$ (\cite{She95}); it is open whether the former inclusion is strict. We also have the following property extending simplicity:

\begin{definition}
A theory $T$ is $\mathrm{NTP}_{2}$ (that is, does not have the \emph{tree property of the second kind}) if there is no array $\{b_{ij}\}_{i, j \in \omega}$ and formula $\varphi(x, y)$ so that there is some fixed $k$ so that, for all $i$, $\{\varphi(x, b_{ij})\}_{j \in \omega}$ is inconsistent, but for any $\sigma \in \omega^{\omega}$, $\{\varphi(x, b_{i\sigma(i)})\}_{(i) \in \omega}$ is consistent.
\end{definition}

The following generalization of Kim-independence beyond the $\mathrm{NSOP}_{1}$ case, \textit{Conant-independence},  was introduced in \cite{GFA}. There, any theory where \textit{Conant-independence} is symmetric was shown to be $\mathrm{NSOP}_{4}$, and Conant-independence was characterized in most of the known examples of $\mathrm{NSOP}_{4}$ theories, leaving open the question of whether all $\mathrm{NSOP}_{4}$ theories, in analogy with $\mathrm{NSOP}_{1}$ theories and Fact \ref{symm}, have symmetric Conant-independence.

\begin{definition}
Let $M$ be a model and $\varphi(x, b)$ a formula. We say $\varphi(x, b)$ \emph{Conant-divides} over $M$ if for \emph{every} invariant Morley sequence $\{b_{i}\}_{i \in \omega}$ over $M$ starting with $b$, $\{\varphi(x, b_i)\}_{i \in \omega}$ is inconsistent. We say $\varphi(x, b)$ \emph{Conant-forks} over $M$ if and only if it implies a disjunction of formulas Conant-dividing over $M$. We say $a$ is \emph{Conant-independent} from $b$ over $M$, written $a \ind^{K^{*}}_{M}b$, if $\mathrm{tp}(a/Mb)$ does not contain any formulas Conant-forking over $M$.
\end{definition}

Note that by Fact \ref{kimlemma}, Conant-independence really does coincide with Kim-independence in $\mathrm{NSOP}_{1}$ theories.

Finally, an additional property is required to complete our classification-theoretic account of generic expansions. It is related to the \textit{dividing order} from \cite{BYC07}, and implies that Conant-forking coincides with Conant-dividing.

\begin{definition}
We say a theory $T$ has the \emph{strong witnessing property} if for $M \prec \mathbb{M}$ there is some sufficiently saturated $\mathbb{M}_{1} \succ M$ (lying in a, say, a very large elementary extension of $\mathbb{M}$) with the following property:

For $b \subset \mathbb{M}_{1}$, $\mathrm{tp}(b/\mathbb{M})$ is an $M$-invariant type such that, if a Morley sequence in that type witnesses Kim-dividing of a formula $\varphi(x, b)$ over $M$, then any any Morley sequence in $\mathrm{tp}(b/M)$ witnesses Kim-dividing of $\varphi(x, b)$ over $M$.
\end{definition}

The following is Theorem 3.15 of \cite{GFA} (relying on the the arguments of \cite{Co15} and \cite{NSOP2}; see also footnote 1 of \cite{INDNSOP3}, and \cite{Lee22}), Theorem 3.16 of \cite{GFA} (using the argument of \cite{Co15}), and Theorem 6.2 of \cite{GFA} (though the arguments of Theorem 4.4 of \cite{Co15}, based on arguments originally due to \cite{Pat06} will suffice in this case):

\begin{fact}\label{conant}
    Let $T$ have the strong witnessing property, and let Conant-independence be symmetric over models. Then $T$ is $\mathrm{NSOP}_{4}$ either simple or $\mathrm{TP}_{2}$, and either $\mathrm{NSOP}_{1}$ or $\mathrm{SOP}_{3}$.
    \end{fact}

We now give a general context for relational expansions of $\mathrm{NSOP}_{1}$ theories with free amalgamation. Let $T$ be a theory with quantifier elimination in a language $\mathcal{L}$ and let $\mathcal{L^{*}}$ be a relational expansion of $\mathcal{L}$. Let $T_{*}$ be a $\mathcal{L}^{*}$-theory expanding $T$. We assume that its model companion $T^{*}$ exists. We assume (1), and either (2) or (2').

(1) Quantifier elimination up to finite covers: Let $A$, $B$ be algebraically closed in the sense of $T$ and have the same quantifier-free $\mathcal{L}^{*}$-type. Then they have the same $\mathcal{L}^{*}$-type.

(2) Let $A, B$ be algebraically closed sets in $T$ and $M \models T$ with $M \subseteq A, B.$ Suppose $A \ind^{K}_{M} B$, and expand $\mathrm{acl}(AB)$ to a $\mathcal{L}^{*}$-structure restricting to a model of $(T_{*})_{\forall}$ (the theory of substructures of $T_{*}$) on $A$ and $B$, and with no new relations from $\mathcal{L}^{*}$ other than those entirely lying in $A$ or $B$ (that is, with $A$ and $B$ \textit{freely amalgamted} over $M$.) Then this expansion of $\mathrm{acl}(AB)$ is another model of $(T_{*})_{\forall}$.

(2$'$) The language $\mathcal{L}^{*}$ consists of $\mathcal{L}$ together with an additional binary relation symbols $E$. Let $A, B$ be algebraically closed sets in $T$ and $M \models T$ with $M \subseteq A, B.$ Suppose $A \ind^{K}_{M} B$, and expand $\mathrm{acl}(AB)$ to a $\mathcal{L}^{*}$-structure restricting to a model of $(T_{*})_{\forall}$ (the theory of substructures of $T_{*}$) on $A$ and $B$, and such that $A/E \cap B/E = M/E$ and each element of $\mathrm{acl}(AB)/(A \cup B)$ is isolated in its own $E$-equivalence class. Then this expansion of $\mathrm{acl}(AB)$ is another model of $(T_{*})_{\forall}$.

When $R$ is an fiber-algebraic definable ternary relation in $T$, and $T^{R}$ exists, $T^{R}$ satisfies both (1), and either (2) (predicate version) or (2') (equivalence relation version). The property (1) is Proposition \ref{qe2}, and the property (2) follows from the proof of Lemma \ref{qe1}.

\begin{example}\label{kimsets}
In every known $\mathrm{NSOP}_{1}$ theory including the simple theories, every type \textit{over a set} has a nonforking extension; under this condition, \cite{DKR22} extend Kim-independence to types over arbitrary sets. (See Example C.2 of \cite{KTW22}). Then, defining the free amalgamation property (2), (2') over arbitrary sets analogously, quantifier elimination (1) follows from either of these properties.
\end{example}

We would like to characterize Conant-independence under these assumptions. The use of ``base monotone" versions of the chain condition or the independence theorem in $\mathrm{NSOP}_{1}$ theories to develop the theory of independence in generic expansions of those theories is not new; see \cite{KR18} and \cite{KTW22}. While the ``algebraically reasonable chain condition" from \cite{KR18} suffices for the case where $\neg R$ is a fiber-algebraic definable ternary relation, we use the following result of \cite{KR19} to indicate the full reach of the inheritance of Kim-independence as Conant-independence under expansions.  Recall that a \textit{Morley sequence} in any ternary relation $\ind^{*}$ over $A$ is an $A$-indiscernible sequence $\{b_{i}\}_{i \in I}$ so that $b_{i} \ind^{*}_{A} b_{< i}$.

\begin{fact}\label{morley1} (\cite{KR19}, Proposition 6.5)
Let $T$ be $\mathrm{NSOP}_{1}$ and $M'\ind^{K}_{M} b$ with $M \prec M'$. Let $I=\{b_{i}\}_{i < \omega}$ be an invariant Morley sequence over $M$ starting with $b$. Then we can find $I' \equiv_{Mb} I$ with $M' \ind^{K}_{M} I$ and $I'$ an $\ind^{K}$-Morley sequence over $M'$. 

\end{fact}

We show that $I'$ has the necessary ``algebraic reasonability" properties. See Theorem 2.21 of \cite{KR18} for a related result proven using similar techniques, and Theorem C.15 of \cite{KTW22} for a result that would work in place of these facts in the case where $\ind^{K}$ is defined over sets (Example \ref{kimsets}).

\begin{fact}\label{morley2}
    In the above fact, let $I' = \{b'_{i}\}_{i < \omega}$. Then for any $i < \omega$ $\mathrm{acl}(MI') \cap \mathrm{acl}(b'_{i}M')=\mathrm{acl}(b'_{i}M)$.
\end{fact}

\begin{proof}
    By compactness, it suffices to prove this when we replace $\omega$ with some large $\kappa$ (say, $\kappa > 2^{|T| + |\mathrm{acl}(M'b)|}$). Assume without loss of generality that $i = 0$; that is, we show $\mathrm{acl}(MI') \cap \mathrm{acl}(b'_{0}M')=\mathrm{acl}(b'_{0}M)$. Suppose $\mathrm{acl}(Mb'_{0}, \ldots, b'_{n})$ meets $\mathrm{acl}(M'b'_{0})$ outside of $\mathrm{acl}(Mb'_{0})$. Let $\{\bar{b}'_{j}\}_{j < \kappa}=\{b_{1+ (jn)}\ldots b_{1+(jn+(n-1))}\}_{j < \kappa}$ be the concatenation into blocks of size $n$ of the sequence $\{b'_{i}\}_{1 \leq i < \kappa}$. Then $\mathrm{acl}(Mb'_{0} \bar{b}'_{j})$ will, by indiscernibility of $I'$ over $M$, meet $\mathrm{acl}(M'b'_{0})$ outside of $\mathrm{acl}(Mb'_{0})$. But the $\mathrm{acl}(Mb'_{0} \bar{b}'_{j})$ will meet pairwise only in $\mathrm{acl}(Mb'_{0})$. So it is impossible for each of the $\mathrm{acl}(Mb'_{0} \bar{b}'_{j})$ to meet $\mathrm{acl}(M'b'_{0})$ outside of $\mathrm{acl}(Mb'_{0})$, as $\kappa$ is too large, a contradiction.
\end{proof}

We finally need the following fact, a strengthening of Kim's lemma, Fact \ref{kimlemma}:

\begin{fact}\label{kimkimlemma}(\cite{KR19}, Fact 5.1)
    Let $T$ be $\mathrm{NSOP}_{1}$, and let $\varphi(x, b)$ Kim-divide over M, and let $\{b_{i}\}_{i < \omega}$ be an $\ind^{K}$-Morley sequence starting with $b$. Then $\{\varphi(x, b_{i})\}_{i < \omega}$ is inconsistent.
\end{fact}

We now characterize Conant-independence in $T^{*}$, when $T^{*}$ satisfies both (1), and (2) or (2').

\begin{proposition} Let $T$ be any $\mathrm{NSOP}_{1}$ theory. Under assumptions (1) and (2) on the model companion $T^{*}$ of an expansion--quantifier elimination up to finite covers, and the free amalgamation property--Conant-independence is the relation $\ind^{T}$ over models of $T^{*}$ inherited from the underlying Kim-independence of $T$ (so in particular, is symmetric.) Moreover, $T^{*}$ has the strong witnessing property.

Under assumption (2'), the same is true, but where the relation $\ind^{T}$ is defined so that $A \ind^{T}_{M} B$ if and only if $A \ind_{M}^{K} B$ in $T$, and $\mathrm{acl}(AM)/E \cap \mathrm{acl}(BM)/E = M/E$.
\end{proposition}

\begin{proof}
We first show one direction of the implication:

\begin{claim}\label{conant1}
    In $T^{*}$, if $a\ind_{M}^{K^{*}} b$ then $a\ind^{T}_M b$.
\end{claim}

\begin{proof}
    Suppose that $a\ind_{M}^{K^{*}} b$ but $a\nind^{T}_M b$. We first show the following claim:

    \begin{subclaim}\
        Let $\{b_{i}\}_{i < \omega}$ be an invariant Morley sequence over $M$ in the sense of $T^{*}$. Then it is an $\ind^{K}$-Morley sequence in the sense of $T$.
    \end{subclaim}

    \begin{proof}
        Invariant Morley sequences are not preserved under taking reducts, but invariant Morley sequences in a \textit{finitely satisfiable} type are; we make use of this point.
        
        It suffices to show that, if we assume that $\mathrm{tp}^{\mathcal{L}^{*}}(c/Md)$ extends to an $M$-invariant global type, then $c \ind_{M}^{K} d$ in the sense of $T$. It follows from the assumption that there is an invariant Morley sequence $\{d_{i}\}_{i < \omega}$ over $M$ starting with $d$ in an $M$-\textit{finitely satisfiable} type, that is indiscernible over $Mc$ in the sense of $T^{*}$. It is then indiscernible over $Mc$ in the sense of $T$, and is an $M$-invariant Morley sequence in an $M$-fintely satisfiable type in the sense of $T$. Therefore, by Fact \ref{kimlemma}, $c \ind_{M}^{K} d$ in the sense of $T$.
    \end{proof}

    Now with (2) by $a\nind^{T}_M b$ and \ref{kimkimlemma}, there is a formula $\varphi(x, b) \in \mathrm{tp}_{\mathcal{L}}(a/Mb)$ so that, for \textit{any} $\ind^{K}$ Morley sequence $\{b_{i}\}_{i \in \omega}$ in the sense of $\mathcal{L}$ starting with $b$, $\{\varphi(x, b_{i})\}_{i \in \omega}$ is inconsistent.  But by the subclaim, every invariant Morley sequence in the sense of $\mathcal{L}^{*}$ is in particular such a sequence. So $\varphi(x, b)$ Conant-divides over $M$, contradiction.

    With (2'), we have the additional case that $\mathrm{acl}(aM)/E \cap \mathrm{acl}(bM)/E \neq M/E $. So in $(T^{*})^{\mathrm{eq}}$, $\mathrm{acl}(aM) \cap \mathrm{acl}(bM) \neq M$, and there is a stable formula in $\mathrm{tp}(a/Mb)$ dividing over $M$. So it divides over $M$ with respect to every invariant Morley sequence, and $a\nind_{M}^{K^{*}} b$

\end{proof}

The following will essentially give the other direction:

\begin{claim}\label{conant2}
Let $a \ind^{T}_{M} b$ and $I=\{b_{i}\}_{i\in \omega}$ be a Morley sequence over $M$ with respect to the free amalgamation given in (2) or (2'), invariant over $M$ in the sense of $\mathcal{L}$ and starting with $b$. Then there is some $I'\equiv^{\mathcal{L}^{*}}_{Mb} I$ indiscernible in the sense of $\mathcal{L}^{*}$ over $Ma$ (with $a \ind^{T}_{M} I'$.)

\end{claim}

\begin{proof}
    By the discussion following Fact \ref{kimlemma}, $\ind^{K}$ and thus $\ind^{T}$ satisfies the extension property. Noting that $\ind^{T}$ is also symmetric, we can find some $\mathcal{L}$-elementary extension $M'$ of $M$ containing $Ma$ so that $M' \ind_{M}^{T} b$. So by replacing $a$ with $M'$, we may assume $a = M' \models T$ is an $\mathcal{L}$-elementary extension of $M$. (It could even have been an $\mathcal{L}^{*}$-elementary extension, but we do not need this.) Note that $M' \ind^{T}_{M} b$ implies $M' \ind^{K}_{M} b$  in the sense of $T$. So in the reduct to $T$, we can choose $I'=\{b'_{i}\}_{i \in I}$ as in Fact \ref{morley1}; that is, some $I' \equiv^{\mathcal{L}}_{Mb} I$ with $M' \ind^{K}_{M} I$ and $I'$ an $\ind^{K}$-Morley sequence over $M'$, in the sense of $T$.

    We will find an expansion of $\mathrm{acl}_{\mathcal{L}}(M'I')$ satisfying $(T^{*})_{\forall}$ so that $\mathrm{acl}_{\mathcal{L}}(I'M)\equiv^{\mathcal{L}^{*}-\mathrm{qf}}_{Mb} \mathrm{acl}_{\mathcal{L}}(IM)$ and so that each $\mathrm{acl}_{\mathcal{L}}(b'_{i}M')$ realizes $\mathrm{qftp}^{\mathcal{L}^{*}}(\mathrm{acl}_{\mathcal{L}}(bM'))$. By the fact that $T^{*}$ is the model companion of $(T^{*})_\forall$, we can then take $I'$ to lie in a monster model of $T^{*}$. Then by (1)--quantifier elimination up to finite covers--$I'\equiv^{\mathcal{L}^{*}}_{Mb} I$ and each $b'_{i}$ realizes $\mathrm{tp}^{\mathcal{L}^{*}}(b/M')$. This will be enough, as we then can extract an $M'$-indiscernible sequence in the sense of $\mathcal{L}^{*}$ by Ramsey and compactness. 
    
    Since the $\mathrm{acl}(b'_{i}M')$ form an $\ind^{K}$-Morley sequence over $M'$, if (2) holds, by repeated applications of (2) we can expand the structure on $\mathrm{acl}_{\mathcal{L}}(M'I')$ so that $\mathrm{qftp}^{\mathcal{L}^{*}}(\mathrm{acl}(b'_{i}M'))=\mathrm{qftp}^{\mathcal{L}^{*}}(\mathrm{acl}(bM'))$, and introduce no further relations. If (2') holds, by repeated applications of (2'), we can expand the structure on $\mathrm{acl}_{\mathcal{L}}(M'I')$ so that $\mathrm{qftp}^{\mathcal{L}^{*}}(\mathrm{acl}(b'_{i}M'))=\mathrm{qftp}^{\mathcal{L}^{*}}(\mathrm{acl}(bM'))$, and take the finest equivalence relation satisfying this requirement. In either case, (2) or (2') will have told us that $\mathrm{acl}_{\mathcal{L}}(M'I')$ has been expanded to a model of $(T_{*})_\forall$. By construction, each $\mathrm{acl}_{\mathcal{L}}(b'_{i}M')$ realizes $\mathrm{qftp}^{\mathcal{L}^{*}}(\mathrm{acl}(bM'))$. 
    
    So it remains to show that $\mathrm{acl}_{\mathcal{L}}(I'M)\equiv^{\mathcal{L}^{*}-\mathrm{qf}}_{Mb} \mathrm{acl}_{\mathcal{L}}(IM)$. By \ref{morley2}, $\bigcup_{i =0}^{\omega} \mathrm{acl}_{\mathcal{L}}(M'b'_{i}) \cap \mathrm{acl}_{\mathcal{L}}(IM) = M$. Under (2), this means that, by not introducing any relations outside of the $\mathrm{acl}_{\mathcal{L}}(M'b'_{i})$, we introduced no relations on $\mathrm{acl}_{\mathcal{L}}(IM)$ that were not already on one of the $\mathrm{acl}(Mb'_{i})$. So the $\mathrm{acl}_{\mathcal{L}}(Mb'_{i})$, which by construction have the same quantifier-free type in $\mathcal{L}^{*}$ as the $\mathrm{acl}_{\mathcal{L}}(Mb'_{i})$, are in fact freely amalgamated over $M$. Therefore, $\mathrm{acl}_{\mathcal{L}}(I'M)\equiv^{\mathcal{L}^{*}-\mathrm{qf}}_{Mb} \mathrm{acl}_{\mathcal{L}}(IM)$. Under (2'), since $M' \ind_{M}^{T} b$, $M'/E\cap \mathrm{acl}_{\mathcal{L}}(Mb'_{i})/E = M/E$. So because $\mathrm{acl}_{\mathcal{L}}(Mb'_{i})/E \cap \mathrm{acl}_{\mathcal{L}}(Mb'_{j})/E \subseteq \mathrm{acl}_{\mathcal{L}}(M'b'_{i})/E \cap \mathrm{acl}_{\mathcal{L}}(M'b'_{j})/E= M'/E$, $\mathrm{acl}_{\mathcal{L}}(Mb'_{i})/E \cap \mathrm{acl}_{\mathcal{L}}(Mb'_{j})/E = M/E$. Moreover, $(\bigcup_{i =0}^{\omega} \mathrm{acl}_{\mathcal{L}}(M'b'_{i}) ) \cap \mathrm{acl}_{\mathcal{L}}(IM) = M$ shows that, isolating each element of $\mathrm{acl}_{\mathcal{L}}(I'M')$ outside of $\bigcup_{i =0}^{\omega} \mathrm{acl}_{\mathcal{L}}(M'b'_{i}) $, we have isolated each element of $\mathrm{acl}_{\mathcal{L}}(I'M)$ outside of $\bigcup_{i =0}^{\omega} \mathrm{acl}_{\mathcal{L}}(Mb'_{i}) $. So again, the $\mathrm{acl}_{\mathcal{L}}(Mb'_{i})$, are in fact freely amalgamated over $M$.

\end{proof}

Now we show the strong witnessing property. A Morley sequence with respect to free amalgamation over $M$, invariant over $M$ in the sense of $\mathcal{L}$, will by quantifier elimination also be invariant over $M$ in the sense of $\mathcal{L}^{*}$. So it suffices to show that if such a Morley sequence witnesses dividing of an $\mathcal{L}^{*}$-formula $\varphi(x, b)$ over $M$, then every invariant Morley sequence in the sense of $\mathcal{L}^{*}$ witnesses dividing of $\varphi(x, b)$ over $M$. Suppose not. Then there is some invariant Morley sequence in the sense of $\mathcal{L}^{*}$ $J=\{b_{i}\}_{i < \omega}$ starting with $b$ with $\{\varphi(x, b_{i})\}_{i < \omega}$ consistent. By Ramsey and compactness, choose $a$ realizing this set, so that $J$ is indiscernible over $Ma$. Then $\mathrm{tp}(a/Mb)$ contains no formulas Conant-dividing over $M$. So by the proof of Claim \ref{conant1}, $a \ind_{M}^{T}b$. By Claim \ref{conant2}, this contradicts the fact that $\varphi(x, b)$ divides with respect to a free amalgamation Morley sequence, invariant over $M$ in the sense of $\mathcal{L}$.

By the strong witnessing property, Conant-forking coincides with Conant-dividing. Since the Morley sequences considered in Claim \ref{conant2} are invariant over $M$ in the sense of $\mathcal{L}^{*}$, it follows from that claim that $\ind^{T}$ implies $\ind^{K^{*}}$. With Claim \ref{conant1}, this gives $\ind^{T} = \ind^{K^{*}}$

\end{proof}

By the previous proposition and \ref{conant}, we get the following corollary.

\begin{corollary}\label{classif}
Let $T$ be any $\mathrm{NSOP}_{1}$ theory. If the model companion $T^{*}$ of an expansion satisfies both (1), and (2) or (2'), it is either simple or $\mathrm{TP}_{2}$, and is either $\mathrm{NSOP}_{1}$ or strictly $\mathrm{NSOP}_{4}$; moreover, $\ind^{K^{*}}= \ind^{T}$.

\end{corollary}

If $T^{*}$ is $\mathrm{NSOP}_{1}$, this implies $\ind^{K} =\ind^{K^{*}} = \ind^{T}$. On the other hand, if $\ind^{K} = \ind^{T}$, then $\ind^{K}$ is symmetric, so $T^{*}$ is $\mathrm{NSOP}_{1}$ by Fact \ref{symm}.

In the case where $T^{*} = T^{R}$, our main result will be to characterize the $\mathrm{NSOP}_{1}$ case. We will use the Kim-Pillay characterization of $\mathrm{NSOP}_{1}$, from \cite{KR17}, to obtain an abstract criterion for $T^{*}$ to be $\mathrm{NSOP}_{1}$ in terms of $\ind^{T}$.

\begin{fact}\label{KP}
Let $T$ be any theory.

(1a) Let $\ind^{K^{*}}$ be symmetric and satisfy the independence theorem: for $a_{1} \ind^{K^{*}}_{M} b_{1}$, $a_{2} \ind^{K^{*}}_{M} b_{2}$, $b_{1} \ind^{K^{*}}_{M} b_{2}$, and $a_{1} \equiv^{\mathcal{L}^{*}}_{M} a_{2}$, there is some $a \ind^{K^{*}}_{M} b_{1}b_{2}$ with $a \equiv_{Mb_{i}}a_{i}$ for $i=1, 2$. Then $T$ is $\mathrm{NSOP}_{1}$.

(1b) If $T$ is $\mathrm{NSOP}_{1}$, then $\ind^{K^{*}}=\ind^{K}$ satisfies the independence theorem.

(2) Let $T$ be $\mathrm{NSOP}_{1}$. Then $T$ is simple if and only if $\ind^{K^{*}}= \ind^{K}$ satisfies base monotonicity: $a \ind^{K}_{M} B$ and $M \preceq M' \subseteq B$ implies $a \ind^{K}_{M'} B$.

\end{fact}

\begin{proof}
    (1a) follows from the definition of $\ind^{K^{*}}$, and Theorem 9.1 of \cite{KR17}. (We did not actually need symmetry here, and could have given a proof using \ref{kimlemma}.) (1b) is Theorem 6.5 of \cite{KR17}. (2) is Proposition 8.8 of \cite{KR17}.
\end{proof}

Note that, when $T$ is stable and $T^{R}$ is well-defined, $\ind^{T}$ is base-monotone in the predicate version. So from Corollary \ref{classif} and Fact \ref{KP}, we get the following:

\begin{lemma}\label{reduction1}
Let $T$ be weakly minimal and let $R$ be a ternary relation definable in $T$. Assume $\neg R$ is fiber-algebraic. Then the equivalence relation version of $T^{R}$ is $\mathrm{NSOP}_{1}$ if and only if $\ind^{T}$ satisfies the independence theorem. Otherwise, $T^{R}$ is $\mathrm{TP}_{2}$ and strictly $\mathrm{NSOP}_{4}$.

For the predicate version of $T^{R}$, this is the same, but replace ``$\mathrm{NSOP}_{1}$" with ``simple."

\end{lemma}

We will use this criterion in the proof of our main result, Theorem \ref{main}, to translate the classification-theoretic properties of $T^{R}$ into properties of $T$. 

We conclude by showing that in the equivalence relation version, $T^{R}$, when $\mathrm{NSOP}_{1}$, is usually strictly $\mathrm{NSOP}_{1}$.

\begin{proposition}
    (Equivalence relation version) Let $T$ be weakly minimal and let $R$ be a ternary relation that is definable in $T$, and assume that $\neg R$ is fiber-algebraic. Suppose that in $T$ there are $a, b, A$ so that $\mathrm{acl}(Aab) \neq \mathrm{acl}(Aa) \cup \mathrm{acl}(Ab)$ (i.e. $T$ has nontrivial pregeometry). Then $T^{R}$ is not simple.
    
    \end{proposition}

    \begin{proof}
       Let $\ind$ be forking-independence in $T$. First of all, let us choose $M \ind_{A} ab$ so that $M$ can be expanded to a model of $T^{R}$. By properties of forking in weakly minimal theories, we still have $\mathrm{acl}(Mab) \neq \mathrm{acl}(Ma) \cup \mathrm{acl}(Mb)$, so we can assume $M = A$ can be expanded to a model of $T^{R}$. 
       
       Now choose $M' \ind_{M} a$ containing $b$ so that $M'$ can be expanded to a model of $T^{R}$ elementarily extending $M$. Choose some $c \in \mathrm{acl}(Mab) \backslash \mathrm{acl}(Ma) \cup \mathrm{acl}(Mb)$. Again by the properties of forking, $c \in \mathrm{acl}(aM') \backslash \mathrm{acl}(aM) \cup M'$.

       Now choose $d \ind_{M} M'a$. So $a \ind_{M} M'd$ by properties of forking.

       Choose some $M'' \supseteq \mathrm{acl}(adM')$ with $M'' \models T$. Expand $M''$ to an $\mathcal{L}^{E}$-structure by defining $E$ as follows:

       \begin{itemize}
           \item On $M'$, $E$ is defined so that $(M', E(M'))$ is a model of $T^{R}$ elementarily extending $M$

           \item All elements of $M'' \backslash M'$ are isolated in their own $E$-class, except for $a, d$, which are in their own class of size $2$.
       \end{itemize}

       Since there are no $E$-equivalence classes with three distinct elements that are not entirely inside $M'$, which is a model of $T^{R}$, $M''$ is a model of $T_{R}$. So it can be identified with a substructure of a model of $T^{R}$.

       Now $a \ind_{M} dM'$ and $\mathrm{acl}(aM)/E \cap \mathrm{acl}(dM')/E = M/E$ by construction. So $a \ind^{T}_{M} dM'$. On the other hand $c/E \in (\mathrm{acl}(aM')/E \cap \mathrm{acl}(dM')/E) \backslash M'/E$. So  $a \nind^{T}_{M'} dM'$ and $\ind^{T}$ is not base monotone. So by Fact \ref{KP}.2, $T^{R}$ is not simple.

    \end{proof}

    In fact, \textit{in the equivalence relation version}, when $T^{R}$ is simple, so $T$ does not satisfy the hypothesis of this proposition (is geometrically trivial), $\ind^{E}$ must be stationary. So \textit{in the equivalence relation version}, if $T^{R}$ is simple, $T^{R}$ is stable.

\section{The group configuration theorem}

We now prove the main result of this paper, \ref{main}. \textit{Throughout this section, we assume the hypotheses of this theorem}: $T$ is weakly minimal theory with quantifier elimination, $R$ is a ternary relation definable in $T$, and $\neg R$ is fiber-algebraic. So because $T$ is $\mathrm{nfcp}$, Proposition \ref{mc} says that $T^{R}$ exists.

We first state the basic amalgamation property for algebraically closed sets in stable theories. This is just the classical independence theorem in the model companion of theories with a generic predicate or equivalence relation.

\begin{fact} (\cite{CP98}, \cite{KR18}).
    Let $T$ be a stable theory with quantifier elimination in the language $\mathcal{L}$, and let $\mathcal{L}^{*}$ be an expansion of $\mathcal{L}$ by a predicate symbol $P$ or a binary relation symbol $E$. For $M \models T$ and $1 \leq i \neq j \neq k \leq 3$, $j < k$, let $p_{i}(X_{j}, X_{k}, X_{jk})$ be quantifier-free $\mathcal{L}^{*}(M)$-types over $M$ consistent with $\mathrm{diag}_{T}(M)$, so that for $A'_{i}, A'_{j}, A'_{jk} \models p_{i}(X_{j}, X_{k}, X_{jk})$, $A'_{i}, A'_{j}, A'_{jk}$ are algebraically closed sets in the sense of $\mathcal{L}$, $A'_{jk} = \mathrm{acl}_{\mathcal{L}}(A'_{j}A'_{k}) \backslash (A'_{j} \cup A'_{k})$, $A'_{j} \ind_{M} A'_{k}$ (forking-independence in the sense of $T$), and if the expansion is by $E$, $A'_{j}/E \cap A'_{k}/E = M/E$. Assume compatibility of these pairs: for $1 \leq i \neq j \neq k \leq 3$, $p_i|_{X_{j}} = p_{j}|_{X_{k}}$. Then in a monster model of $T$, there are forking-independent $A_{1}, A_{2}, A_{3}$ over $M$, and there is an interpretation of $P$ or $E$ on $\mathrm{acl}_{\mathcal{L}}(A_{1}A_{2}A_{3})$ so that for $1 \leq i \neq j \neq k \leq 3$, $j < k$, $\mathrm{acl}_{\mathcal{L}}(A_{j}A_{k}) \models p_{i}$, and moreover $A_{1}/E \cap \mathrm{acl}_{\mathcal{L}}(A_{2}A_{3})/E = M/E$ when the expansion is by $E$.
\end{fact}

\begin{proof}
    In the expansion by $P$, this is just the content of 2.7 of \cite{CP98} (which was proven for simple theories). In the expansion by $E$, this is the content of Theorem 4.5 of \cite{KR18} (which was proven for $\mathrm{NSOP}_{1}$ theories), where $L'$ in the statement of that theorem is taken to be $L$ together with a symbol for a unary function to a new sort.
\end{proof}

\begin{lemma}
    In the previous fact, if the expansion is by $P$, the interpretation of $P$ can be chosen to contain no points of $\mathrm{acl}_{\mathcal{L}}(A_{1}A_{2}A_{3}) \backslash \bigcup_{1 \leq j < k \leq 3}\mathrm{acl}_{\mathcal{L}}(A_{j}A_{k})$. If the expansion is by $E$, each point of that set can be assumed isolated in its own $E$-class.
\end{lemma}

\begin{proof}
    If we change the interpretation of $P$ or $E$ outside of $\mathrm{acl}_{\mathcal{L}}(A_{j}A_{k})$, so that this requirement is met, this does not change the fact that $\mathrm{acl}_{\mathcal{L}}(A_{j}A_{k}) \models p_{i}$, nor that $A_{1}/E \cap \mathrm{acl}_{\mathcal{L}}(A_{2}A_{3})/E = M/E$.

    (Note that in \cite{CP98}, this is part of the proof of the previous fact, while in \cite{KR17}, it is stated in the proof that the map can indeed be defined arbitrarily outside of $\bigcup_{1 \leq j < k \leq 3}\mathrm{acl}_{\mathcal{L}}(A_{j}A_{k})$.)
\end{proof}

We first prove the following lemma, reducing the classification-theoretic properties of the expansion $T^{R}$ to the structure of the original weakly minimal theory $T$. As usual in a stability-theoretic context, independence, denoted $\ind$, is forking-independence in the sense of $T$.

\begin{lemma} \label{reduction2}
The theory $T^{R}$ is $\mathrm{SOP}_{3}$ if and only if in $T$, there are algebraically closed $A \subseteq A_{1}, A_{2}, A_{3}$, the $A_{i}$ independent over $A$, and $a_1, a_{2}, a_{3}$ with $\models \neg R(a_{1}, a_{2}, a_{3}) $, so that for $1 \leq i, j, k \leq 3$ distinct, $a_{i} \in \mathrm{acl}(A_{j}A_{k}) \backslash (A_{i} \cup A_{j})$
\end{lemma}

\begin{proof}
       ($\Rightarrow$) First suppose $T^{R}$ is $\mathrm{SOP}_{3}$. So $\ind^{T}$, by Lemma \ref{reduction1}, does not satisfy the independence theorem. So for $1 \leq i, j, k \leq 3$ distinct, $j < k$, there are compatible types $p_{i}(X_{j}, X_{k})$ in $T^{R}$ of $\ind^{T}$-independent pairs over some $M \models T^{R}$ that do not have a common realization by $A'_{1}, A'_{2}, A'_{3}$ with $A'_{1} \ind^{T} A'_{2}A'_{3}$. Let $M \models T$, $A_{1}, A_{2}, A_{3} \supseteq M$, together with an interpretation of $P$ or $E$ on $\mathrm{acl}(A_{1}A_{2}A_{3})$ be obtained as in the previous lemma, from the types corresponding to $p_{i}(X_{j}, X_{k})$ by quantifier elimination up to finite covers (Proposition \ref{qe2}). In the rest of this proof, let $i, j, k$ range over distinct $1 \leq i, j, k \leq 3 $
       
       Suppose that  $\mathrm{acl}(A_{1}A_{2}A_{3})$ together with this interpretation $P(\mathrm{acl}(A_{1}A_{2}A_{3}))$ or $E(\mathrm{acl}(A_{1}A_{2}A_{3}))$ of $P$ or $E$ satisfies $(T_{R})_\forall$ (is a substructure of a model of $T_{R}$). Then $A_{1}, A_{2}, A_{3}$ could be identified by a common realization $A'_{1}, A'_{2}, A'_{3}$ of the $p_{i}(X_{j}, X_{k})$ with $A'_{1} \ind^{T} A'_{2}A'_{3}$, a contradiction.  So $\mathrm{acl}(A_{1}A_{2}A_{3})$ together with the additional structure does not satisfy $(T_{R})_\forall$. 
       
       To witness this, there are distinct $a_{1}, a_{2}, a_{3}$ belonging to $P(\mathrm{acl}(A_{1}A_{2}A_{3}))$, or belonging to the same $E(\mathrm{acl}(A_{1}A_{2}A_{3}))$-equivalence class, so that $ \models \neg R(a_{1}, a_{2}, a_{3})$. Relabeling, it suffices to show that each of the three $a_{i}$ belongs to precisely one of the three $\mathrm{acl}(A_{j}A_{k})$. Because each of those pairs do satisfy the quantifier-free type of a model of $(T_{R})_{\forall}$, $a_{1}, a_{2}, a_{3}$ cannot all belong to the same $\mathrm{acl}(A_{j}A_{k})$. Because on $\mathrm{acl}(A_{1}A_{2}A_{3}) \backslash \bigcup_{1 \leq j < k \leq 3}\mathrm{acl}(A_{j}A_{k})$, there are no points of $P$ or each point is isolated in its own $E$-class, none of $a_{1}, a_{2}, a_{3}$ belong to $\mathrm{acl}(A_{1}A_{2}A_{3}) \backslash \bigcup_{1 \leq j < k \leq 3}\mathrm{acl}(A_{j}A_{k})$. Finally, it remains to show that no two of $a_{1}, a_{2}, a_{3}$ can belong to $\mathrm{acl}(A_{i}A_{j})$, while a third belongs to a different $\mathrm{acl}(A_{i}A_{k})$ but not to $A_{i}$. Because in $T$, $A_{j} \ind_{A_{i}} A_{k} $, the third cannot be algebraic over the other two, as then it would belong to $\mathrm{acl}(A_{i}A_{j})\cap \mathrm{acl}(A_{i}A_{k})= A_{i}$. But then $\models R(a_{1}, a_{2}, a_{3})$ must hold, as one of the $a_{i}$ is not algebraic over the other two.

       ($\Leftarrow$) Now assume the second condition. By Lemma \ref{reduction1}, it suffices to show that the independence theorem fails for $\ind^{T}$. By taking some $M \models T$, that can be expanded to a model of $T^{R}$, independently from $A_{1}A_{2}A_{3}$ over $A$, and replacing $A_{i}$ with $\mathrm{acl}(MA_{i})$, we can assume $A=M$ is a model of $T$ that can be expanded to a model  $(M, E(M))$ or $(M, P(M))$ of $T^{R}$. In the equivalence relation case, fix some $m \in M$, and expand each of the $\mathrm{acl}(A_{i}A_{j})$ so that the additional structure extends that on $M$, and $a_{k}$ lies in the same equivalence class as $m$, while each point of $\mathrm{acl}(A_{i}A_{j}) \backslash M$ besides $a_{k}$ is isolated in its own class. In the predicate case, instead add no point of $\mathrm{acl}(A_{i}A_{j}) \backslash M$ to the intepretation of $P$, besides $a_{k}$. Because $a_{k}$ is not algebraic over $M$, either of these constructions give a model of $(T_{R})_{\forall}$. So these expansions of $\mathrm{acl}(A_{i}A_{j})$ determine, by the quantifier elimination up to finite covers in $T^{R}$, $\mathcal{L}^{E}$ or $\mathcal{L}^{P}$-types $p_{k}(X_{i}, X_{j})$ over the expansion of $M$ in $T^{R}$ for $i < j$. Because $a_{k} \notin A_{i}\cup A_{j}$, no nontrivial new structure was added to one of the $A_{i}$ in any pair, other than that on $M$. So these types agree on the $X_{i}$, by the quantifier elimination up to finite covers. And by construction, each is realized by a $\ind^{T}$-independent pair. So a failure of the independence theorem for $\ind^{T}$ would be implied, if we can show that these types cannot by jointly realized in $T^{R}$ by a triple that is forking-independent in the sense of $T$.

       We claim that an obstruction to this joint realization would occur if
       
       $$\mathrm{tp}_{\mathcal{L}}(A_{1}A_{2}A_{3}/M) \vdash \forall x_{1} x_{2} x_{3}\bigwedge_{j \neq i \neq k, j < k} \varphi_{i}(X_{j}, X_{k}, x_{i}) \to \neg R (x_{1}, x_{2}, x_{3})$$ for $\varphi_{i}(A_{j}, A_{k}, x_{i})$ a $\mathcal{L}$-formula isolating $\mathrm{tp}_{\mathcal{L}}(a_{i}/A_{j}A_{k})$. Indeed, a joint realization in $T^{R}$ of the $p_{k}(X_{i}, X_{j})$ that is a forking-independent triple in the sense of $T$ over $M$,  $A'_{1}, A'_{2}, A'_{3}$,  must satisfy $\mathrm{tp}_{\mathcal{L}}(A_{1}A_{2}A_{3}/M)$, by stationarity. Therefore, it must satisfy the formula on the right. But because the $A'_{1}, A'_{2}, A'_{3}$ jointly realize the $p_{k}(X_{i}, X_{j})$, for $i, j, k$ there must be some $a'_{i} \models \varphi_{i}(A'_{j}, A'_{k}, x_{i})$ belonging to the $E$-class of $m$, or to the interpretation of $P$. So $a'_{1}, a'_{2}, a'_{3}$ must all belong to the same equivalence class or to the interpretation of $P$. But by the formula on the right, $\models \neg R(a'_{1}, a'_{2}, a'_{3})$. This contradicts the axioms of $T_{R}$.

       So it remains to show that

       $$A_{1}A_{2}A_{3}\models \forall x_{1} x_{2} x_{3}\bigwedge_{j \neq i \neq k, j < k} \varphi_{i}(X_{j}, X_{k}, x_{i}) \to \neg R (x_{1}, x_{2}, x_{3})$$ For $i, j, k$, suppose $a''_{i}$ satisfies $\varphi_{i}(A_{j}, A_{k}, x_{i})$ and let $\sigma_{i}$ be an automorphism of $\mathrm{acl}(A_{j}A_{k})$ over $A_{j}A_{k}$ sending $a_{i}$ to $a''_{i}$. The independence theorem in $T^{A}$, Theorem 3.7 of \cite{CP98} does not say these automorphisms have a common extension--only that some conjugates of these automorphisms do. But the proof of that theorem does in fact show that compatible automorphisms of the algebraic closures of pairs in an independent triple, indeed have a common extension. Since this is not stated explicitly, we review the proof of everything we need; we work in $T$. For our purposes, it suffices to show for each $i, j, k$ that $\sigma_{i}$ as above, so an automorphism of $\mathrm{acl}(A_{j}A_{k})$ over $A_{j}A_{k}$ with $\sigma_{i}(a_{i}) = a''_{i}$, can be chosen so that it extends to an automorphism $\tilde{\sigma}_{i}$ of the monster model $\mathbb{M}\models T$ that is the identity on $\mathrm{acl}(A_{i}A_{j})$ and $\mathrm{acl}(A_{i}A_{k})$. Indeed, then we can compose all three of the $\tilde{\sigma}_{i}$ together, to get an automorphism extending each of the $\sigma_{i}$.  Because $a_{i} \equiv^{\mathcal{L}}_{A_{j}A_{k}} a''_{i}$, we will get the desired automorphism $\tilde{\sigma}_{i}$ of $\mathbb{M}$ over $\mathrm{acl}(A_{i}A_{j}) \mathrm{acl}(A_{i}A_{k})$ with $\tilde{\sigma}_{i}(a_{i}) = a''_{i}$, as long as the orbit of $a_{i}$ over $A_{j}A_{k}$ is the same as that of $a_{i}$ over $\mathrm{acl}(A_{i}A_{j}) \mathrm{acl}(A_{i}A_{k})$. Now the latter orbit, in the sense of $T^{\mathrm{eq}}$, belongs to $\mathrm{dcl}(\mathrm{acl}(A_{i}A_{j}) \mathrm{acl}(A_{i}A_{k})) \cap \mathrm{acl}(A_{j}A_{k})$. Now recall the claim of Theorem 3.7 of \cite{CP98}, namely that $\mathrm{dcl}(\mathrm{acl}(AB) \mathrm{acl}(AC)) \cap \mathrm{acl}(BC)=\mathrm{dcl}(BC)$ for $A, B, C$ independent sets in a stable theory. This claim implies that $\mathrm{dcl}(\mathrm{acl}(A_{i}A_{j}) \mathrm{acl}(A_{i}A_{k})) \cap \mathrm{acl}(A_{j}A_{k})=\mathrm{dcl}(A_{j}A_{k})$. But since the orbit of $a_{i}$ over $\mathrm{acl}(A_{i}A_{j}) \mathrm{acl}(A_{i}A_{k})$  is then in  $\mathrm{dcl}(A_{j}A_{k})$, all of the conjugates of $a_{i}$ over $A_{j}A_{k}$ must belong to the orbit of $a_{i}$ over $\mathrm{acl}(A_{i}A_{j}) \mathrm{acl}(A_{i}A_{k})$, so the orbit of $a_{i}$ over $\mathrm{acl}(A_{i}A_{j}) \mathrm{acl}(A_{i}A_{k})$ must coincide with the orbit of $a_{i}$ over $A_{j}A_{k}$.

       So there is an automorphism $\sigma$ of $\mathbb{M}$ extending $\sigma_{1}, \sigma_{2}, \sigma_{3}$. So $\models \neg R(a_{1}, a_{2}, a_{3})$ implies $\models \neg R(\sigma(a_{1}), \sigma(a_{2}), \sigma(a_{3}))$, so $\models \neg R(\sigma_{1}(a_{1}), \sigma_{2}(a_{2}), \sigma_{3}(a_{3}))$, so  $\models \neg R(a''_{1}, a''_{2}, a''_{3})$.

\end{proof}

This is the main technical lemma required for the group configuration theorem. To prove it, we use properties of forking in weakly minimal theories throughout.

\begin{lemma}\label{reduction3}
In the previous lemma, we can further require that $U(A_{i}/A)=1$ for $1 \leq i \leq 3$, in the sense of $T$.
\end{lemma}

\begin{proof}
Throughout this proof we refer to $T$ and use the notation of the previous lemma. It suffices to find some $A \subseteq D \subset A_{1} \cup A_{2} \cup A_{3}$ with $U(A_{i}/D)=1$ so that the second condition of that lemma is satisfied replacing each $A_{i}$ with $\mathrm{acl}(A_{i}D)$. We do this by handling $A_{1}$, $A_{2}$ and $A_{3}$ successively.

We begin with the following observation: relative to a given set, if $b_{1}, \ldots, b_{n}$ is an independent sequence and $a \in \mathrm{acl}(b_{1}, \ldots b_{n})$, there is some \textit{least} $S \subseteq \{b_{1}, \ldots, b_{n}\}$ so that $a \in \mathrm{acl}(S)$. Because if $S_{1}, S_{2} \in \mathrm{acl}(b_{1}, \ldots b_{n})$ are two \textit{minimal} such sets, then they are independent over $S_{1} \cap S_{2}$, so $a \in \mathrm{acl}(S_{1}) \cap \mathrm{acl}(S_{2})=\mathrm{acl}(S_{1} \cap S_{2})$, contradicting minimality.

Choose $1 \leq i, j, k \leq 3$ distinct. We can assume $A_{i} = \mathrm{acl}_{A}(b_{1}, \ldots, b_{n})$ for $b_{1}, \ldots, b_{n}$ independent over $A$. The $b_{1}, \ldots, b_{n}$ are thus independent over $A_{j}$ and $A_{k}$ since $A_{j} \ind_{M} A_{i}$, $A_{k} \ind_{M} A_{i}$. By the observation above, let $S_{j}, S_{k} \subseteq \{b_{1}, \ldots, b_{n}\}$ be respectively the least so that $a_{j} \in \mathrm{acl}(A_{k}S_{j})$ and $a_{k} \in \mathrm{acl}(A_{j}S_{k})$. We claim that $S_{j} \cap S_{k} \neq \emptyset$. Otherwise, as $A_{i}, A_{j}, A_{k}$ are independent over $M$, and $b_{1}, \ldots, b_{n}$ are an independent subset of $A_{k}$, $A_{k}S_{j} \ind_{A} A_{j}S_{k}$. Therefore, $a_{j} 
 \in \mathrm{acl}(A_{k}S_{j}) \ind_{A_{j}A_{k}} \mathrm{acl}(A_{j}S_{k}) \ni a_{k}$, so $a_{j} \ind_{A_{j}A_{k}} a_{k}$. Now $a_{j} \in \mathrm{acl}(A_{i}A_{k}) \backslash A_{k}$ and $A_{i} \ind_{A_{k}} A_{j}$, so $a_{j} \notin \mathrm{acl}(A_{j}A_{k})$. But because $\models \neg R(a_{1}, a_{2}, a_{3})$, $a_{j}$ is algebraic over $\mathrm{acl}(A_{j}A_{k})a_{k} \supseteq Aa_{i}a_{k}$. So $a_{j}$ and $a_{k}$ are dependent over $\mathrm{acl}(A_{j}A_{k}) $, contradicting $a_{j} \ind_{A_{j}A_{k}} a_{k}$. This proves our claim that $S_{j} \cap S_{k} \neq \emptyset$.

Now let $A'=\mathrm{acl}_A(\{b_{1}, \ldots b_{n}\} \backslash \{b\})$ for some $b \in S_{j} \cap S_{k}$. Then $U (A_{i}/A') =1$.  By choice of $S_{j}$ and $S_{k}$, $a_{k} \in \mathrm{acl}(A_{i}A_{j}) \backslash (A_{i} \cup \mathrm{acl}(A'A_{j}))$ and $a_{j} \in \mathrm{acl}(A_{i}A_{k}) \backslash (A_{i} \cup \mathrm{acl}(A'A_{k}))$. By the same reasoning used to show $a_{j} \notin \mathrm{acl}(A_{j}A_{k})$ above, $a_{i} \notin \mathrm{acl}(A_{i}A_{j}) \cup \mathrm{acl}(A_{i}A_{k})$, so $a_{i} \in \mathrm{acl}(\mathrm{acl}(A'A_{j}) \mathrm{acl}(A'A_{k})) \backslash (\mathrm{acl}(A'A_{j}) \cup \mathrm{acl}(A'A_{k}))$.

So replace $A$ with $A' \subseteq A_{i}$ $A_{j}$ with $\mathrm{acl}(A'A_{j})$ and $A_{k}$ with $\mathrm{acl}(A'A_{k})$. Now repeat what we have done for $A_{i}$ for each of $A_{j}$ and $A_{k}$.

\end{proof}

We are now in a position to prove Theorem \ref{main}. First, suppose $G$ is a rank-one connected group type-definable over a parameter set $A$ defining $R$, which we can assume to be algebraically closed. Let $(a, b, c)$ be a generic of the graph of its operation. Then $a$ and $b$ are independent generics of $G$ over $A$ and $c=ab$. Then (as in the construction of a group configuration from an actual group; see \cite{Bays15} for an overview) we can find independent algebraically closed sets $A_{1}, A_{2}, A_{3}$ containing $A$ with $a_{i} \in \mathrm{acl}(A_{j}A_{k}) \backslash (A_{i} \cup A_{j})$; just find independent generics $d_{1}, d_{2}, d_{3}$ of $G$ over $A$ so that $a=d_{1}d^{-1}_{2}$, $b=d_{2}d^{-1}_{3}$ and $c=d_{1}d^{-1}_{3}$. Now note that, by replacing $a_{1}$, $a_{2}$, and $a_{3}$ by elements individually interalgebraic over $A$, we preserve $a_{i} \in \mathrm{acl}(A_{j}A_{k}) \backslash (A_{i} \cup A_{j})$, so by Lemma \ref{reduction2}, $T^{R}$ is not simple.

In the other direction, suppose $T^{R}$ is not simple. Then we get $a_{i}$, $A_{i}$ as in Lemmas \ref{reduction2}, \ref{reduction3}. To summarize, we have $A_{1}, A_{2}, A_{3}$, $a_{1}, a_{2}, a_{3}$ of rank one over $A$, $a_{i} \in \mathrm{acl}(A_{j}A_{k}) \backslash (A_{i} \cup A_{j})$, $A_{1}, A_{2}, A_{3}$ forming an independent triple over $A$, and $a_{1}, a_{2}, a_{3}$ forming a \textit{dependent} triple over $A$, since $\models \neg R(a_{1},a_{2}, a_{3}) $.

Since they are all of rank one over $A$, we know from the properties of forking in weakly minimal theories that $a_{1}, a_{2}, a_{3}$ together with, for $i, j, k$ distinct, each of $a_{i}, A_{j}, A_{k}$, form the lines of a group configuration (recall Figure 1 above). The conclusion follows from the group configuration theorem, Theorem \ref{gct}.

\begin{example}
We given an example of a simple unstable theory $T$
of $\mathrm{SU}$-rank $1$ and ternary relation $R$ definable in $T$ satisfying the group condition of Theorem 1.3, but with $T^{R}$ still simple. Let $T_{0}$ be the theory of two-sorted structures consisting of a vector space $V$ over a finite field and a two-to-one map $F$ from a set $S$ to $V$, with a symmetric ternary relation $U$ relating, for any three distinct fibers of $F$, exactly one point in each of the fibers. It has a model companion $T$ which can be seen to be supersimple of $\mathrm{SU}$-rank $1$ with the evident quantifier elimination. Now let $R(x_{1}, x_{2}, x_{3})$ be defined on $S$ by $\neg (F(x_{1}) = F(x_{2}) + F(x_{3}) \wedge U(x_{1}, x_{2}, x_{3}))\wedge \bigwedge_{1 \leq i, j, \leq 3} \neg(F(x_{i}) = F(x_{j}))$. The independence theorem still holds in $T^{R}$, which exists (see the appendix), for the relation $a\ind_{M}^{a}b$ given by $\mathrm{acl}(aM)\cap\mathrm{acl}(bM)=M$.

However, there is a group configuration theorem for certain simple theories (\cite{BYTW02}), and the left-to-right direction of Lemma 4.1 as well as Lemma 4.2 only require $\mathrm{SU}$-rank one and not the additional assumption of stability, so when $T^{R}$ exists, but is not simple, we may still get a characterization of $R$ in terms of groups.
\end{example}

\appendix 
\section{Existence of $T^{R}$ for fiber algebraic ternary relations}

In Proposition \ref{mc}, we show that $T^{R}$ exists whenever $T$ is $\mathrm{nfcp}$. When $T$ only eliminates $\exists^{\infty}$, $T^{R}$ still exists in the case where $\neg R$ is an algebraic ternary relation. This may be useful for generalizing the results of this paper to case of $T$ unstable.

\begin{proposition}
 Let $T$ eliminate the quantifier $\exists^{\infty}$and let $R$ be the negation of an algebraic ternary relation definable in $T$. Then $T^{R}$ exists.   
\end{proposition}

\begin{proof}
(Predicate version) Let $M \models T_{R}$ and let $\psi(\overline{y}, x_{1}, \ldots, x_{n})$ be a formula of $\mathcal{L}_{0}$. As in the proof of Theorem 2.4 of \cite{CP98}, we can assume that $\psi(\overline{y}, x_{1}, \ldots, x_{n})$ implies that $x_{1}, \ldots, x_{n}$ are distinct, and it suffices to find, for any $1 \leq k \leq n$, some  $\mathcal{L}^{P}$-formula $\rho(\overline{y})$ independently of $M$ so that $M \models \rho(\overline{m})$ if and only if there is some $\overline{a} \in N$ for $N \models T_{R}$ an extension of $M$ such that $\overline{a}$ does not meet $M$ and such that $N \models \psi(\overline{m}, \overline{a}) \wedge \bigwedge_{1}^{k}P(a_{i}) \wedge \bigwedge_{k+1}^{n} \neg P(a_{i}) $; then the $\forall y (\rho(\overline{y}) \rightarrow \exists \overline{x} \psi(\overline{y}, \overline{x}))$ will still axiomatize when $M$ is existentially closed.

We can assume that for $1 \leq i \leq j \leq k$ and $\sigma \in S_{3}$, there is a constant $k_{ij\sigma}$ so that for any $\overline{m}$, $\psi(\overline{m}, \overline{a})$ implies that $\neg R(\sigma(a_{i}, a_{j}, x))$ has exactly $k_{ij\sigma}$ solutions, since every formula $\psi(\overline{y}, x_{1}, \ldots, x_{n})$ can be written as a disjunction of formulas with this property. Consider the condition $\rho(\overline{m})$ on $\overline{m} \in M$ requiring that for $1 \leq i \leq j \leq k$ and $\sigma \in S_{3}$ there are some $0 \leq l_{ij\sigma} \leq k_{ij\sigma}$ and $e_{ij\sigma}^{1}, \ldots, e_{ij\sigma}^{l_{ij\sigma}} \in M \backslash P(M)$, distinct for fixed $ij\sigma$ so that the following condition $\tau(\overline{m}, \overline{e})$ holds:

There is some $M \prec_{\mathcal{L}} N$ and $a_{1} \ldots a_{n} \in N \backslash M$ and, for $1 \leq i \leq j \leq k$ and $\sigma \in S_{3}$, $f^{1}_{ij\sigma} \ldots f^{k_{ij\sigma}-l_{ij\sigma}}_{ij\sigma} \in N\backslash M$, distinct for fixed $ij\sigma$ and distinct from all of the $a_{1}, \ldots, a_{k} $, so that $N \models \psi(\overline{m}, \overline{a})$, for all $1 \leq i_{1}, i_{2}, i_{3} \leq k $, $N \models R(a_{i_{1}}, a_{i_{2}}, a_{i_{3}})$, and for fixed $1 \leq i \leq j \leq k$ and $\sigma \in S_{3}$, $N \models \neg R(\sigma(a_{i}, a_{j}, a))$ for $a$ any of the $e_{ij\sigma}^{1}, \ldots, e_{ij\sigma}^{l_{ij\sigma}}$ or $f^{1}_{ij\sigma} \ldots f^{k_{ij\sigma}-l_{ij\sigma}}_{ij\sigma}$.

It follows from the following claim, used implicitly in \cite{CP98}, that $\tau(\overline{m}, \overline{e})$ is a definable condition in $\mathcal{L}$:

\begin{claim}
For $\varphi(\overline{x}, \overline{y})$ any $\mathcal{L}$-formula, the set of $\overline{a} \in M$ so that there is $\overline{b}$ in an elementary extension of $M$ not meeting $M$ with $\models \varphi(\overline{b}, \overline{a})$ is definable in $\mathcal{L}$
\end{claim}

\begin{proof}
By elimination of $\exists^{\infty}$, we can apply Lemma 2.3 of \cite{CP98}, which says that the set of $\overline{a} \in M$, so that there is $\overline{b}$ in an elementary extension of $M$ not meeting $\mathrm{acl}(a)$ with $\models \varphi(\overline{b}, \overline{a})$, is definable. But it is well-known that for any $a, b$ and $a \in M$ there is always $b' \equiv_{a} b$ with $\mathrm{acl}(b) \cap M = \mathrm{acl}(a)$. So the set we have defined is in fact our desired set.
\end{proof}

Because $\tau(\overline{m}, \overline{e})$ can be expressed definably in $\mathcal{L}$, $\rho(\overline{m})$ can be expressed definably in $\mathcal{L}^{P}$. We claim that $\rho(\overline{m})$ is as desired. First suppose there is some $\overline{a} \in N$ for $N \models T_{R}$ an extension of $M$ such that $\overline{a}$ does not meet $M$ and such that $N \models \psi(\overline{m}, \overline{a}) \wedge \bigwedge_{1}^{k}P(a_{i}) \wedge \bigwedge_{k+1}^{n} \neg P(a_{i}) $. Then we can let $e_{ij\sigma}^{1}, \ldots, e_{ij\sigma}^{l_{ij\sigma}}$ enumerate the solutions in $M$ to $\neg R(\sigma(a_{i}, a_{j}, x))$--note that they must belong to  $ M \backslash P(M)$--and let $f^{1}_{ij\sigma} \ldots f^{k_{ij\sigma}-l_{ij\sigma}}_{ij\sigma}$ enumerate the solutions in $N \backslash M$ to $\neg R(\sigma(a_{i}, a_{j}, x))$--note that they must be distinct from the $a_{1}, \ldots, a_{k}$. Now suppose $\rho(m)$ holds. It remains to expand $N$ to a model of $T_{R}$ extending $M$ so that $N \models \bigwedge_{1}^{k}P(a_{i}) \wedge \bigwedge_{k+1}^{n} \neg P(a_{i}) $. Note that the $a_{i}$ are distinct; add just the $a_{1}, \ldots, a_{k}$ to the domain of $P$, and no other new elements, to form $P(N)$. We must show that $(N, P(N)) \models T_{R}$; that is, for a triple $\overline{n} \in P(N)$, $N \models R(\overline{n})$. This is clearly the case if all of the coordinates over $\overline{n}$ lie in $N/M$, and also if they all lie in $M$, as we assume that for $1 \leq i_{1}, i_{2}, i_{3} \leq k $, $N \models R(a_{i_{1}}, a_{i_{2}}, a_{i_{3}})$. If two of the coordinates of $\overline{n}$ are $a_{i}$ and $a_{j}$ for $1 \leq i \leq j \leq k$, and another coordinate lies in $M$, then $N \models R(\overline{n})$ still holds as all $k_{ij\sigma}$ of the solutions to any of the $\neg R(\sigma(a_{i}, a_{j}, x))$ must either be one of the $f^{1}_{ij\sigma} \ldots f^{k_{ij\sigma}-l_{ij\sigma}}_{ij\sigma}$ that are not in $M\cup \overline{a}$ or one of the $e_{ij\sigma}^{1}, \ldots, e_{ij\sigma}^{l_{ij\sigma}}$ that are not in $P(N)$. Finally, if exactly two of the coordinates of $\overline{n}$ lie in $M$, then $N \models R(\overline{n})$ as the other coordinate cannot be algebraic over the two that belong to $M$.

(Equivalence relation version) Similar.

\end{proof}

\bibliographystyle{plain}
\bibliography{refs}

\end{document}